\documentclass{article}
\usepackage{mathrsfs,graphicx,framed,color,amssymb,amsmath,amsthm,appendix}
\usepackage[boxed, noline, ruled, linesnumbered]{algorithm2e}
\usepackage[colorlinks, citecolor=blue]{hyperref}
\usepackage{epsfig, graphicx}
\usepackage{authblk}
\usepackage{latexsym,amsfonts,amsbsy,amssymb}
\usepackage{amsmath,amsthm}
\usepackage{enumerate}
\usepackage{cite}
\usepackage{lineno}
\usepackage[total={7in, 9in}]{geometry}
\makeatletter 

\@addtoreset{equation}{section}
\makeatother 
\allowdisplaybreaks 
\newtheorem{theorem}{Theorem}[section]
\newtheorem{lemma}{Lemma}[section]

\newtheorem{remark}{Remark}[section]

\def\cH{\mathcal{H}}

\def\cQ{\mathcal{Q}}
\def\cT{\mathcal{T}}

\begin{document}
\title{A Nonnested Augmented Subspace Method for Kohn-Sham Equation}

\author[1,2]{Guanghui Hu}

\author[3,4]{Hehu Xie}

\author[5]{Fei Xu}

\author[3]{Gang Zhao}

\affil[1]{Department of Mathematics, University of Macau, Macao
  S.A.R., China}

\affil[2]{UM Zhuhai Research Institute, Zhuhai, Guangdong, China}

\affil[3]{LSEC, ICMSEC, Academy of Mathematics and Systems Science,
  Chinese Academy of Sciences, Beijing 100190, China}

\affil[4]{School of Mathematical Sciences, University of Chinese
  Academy of Sciences, Beijing, 100049, China}

\affil[5]{Faculty of Science, Beijing University of Technology,
  Beijing 100124, China}

\date{}
\maketitle
\begin{abstract}
  In this paper, a novel adaptive finite element method is proposed to solve the
  Kohn-Sham equation based on the moving mesh (nonnested mesh) adaptive
  technique and the augmented subspace method.  Different from the classical
  self-consistent field iterative algorithm which requires to solve the
  Kohn-Sham equation directly in each adaptive finite element space, our
  algorithm transforms the Kohn-Sham equation into some linear boundary value
  problems of the same scale in each adaptive finite element space, and then the
  wavefunctions derived from the linear boundary value problems are corrected by
  solving a small-scale Kohn-Sham equation defined in a low-dimensional
  augmented subspace.  Since the new algorithm avoids solving large-scale
  Kohn-Sham equation directly, a significant improvement for the solving efficiency can
  be obtained.  In addition, the adaptive moving mesh technique is used to
  generate the nonnested adaptive mesh for the nonnested augmented subspace
  method according to the singularity of the approximate wavefunctions. The
  modified Hessian matrix of the approximate wavefunctions is used as the metric
  matrix to redistribute the mesh.  Through the moving mesh adaptive technique,
  the redistributed mesh is almost optimal.  A number of numerical experiments
  are carried out to verify the efficiency and the accuracy of the proposed
  algorithm.


\end{abstract}

\section{Introduction}

Density functional theory (DFT) is one of the major breakthroughs in quantum
physics and quantum chemistry. In the framework of density
functional theory, the many-body problem can be simplified as the motion of
electrons without interaction in the effective potential field, which includes
the Coulomb potential of the external potential field, the Hartree potential
generated by the interaction between electrons, and the exchange-correlation
potential for the nonclassical Coulomb interaction. For the exchange-correlation
potential, it has always been a difficulty in density functional theory.  At
present, there is no exact analytical expression for the exchange-correlation
potential.  Generally, it is approximately described by the local density
approximation (LDA), the local spin density approximation (LSDA) and the
generalized gradient approximation (GGA), etc.  The Kohn-Sham equation is one of
the most important models in the calculation of electronic structure, which
transforms the wavefunctions (in $\mathbb R^{3N}$ space) describing $N$ particle
states into density function (in $\mathbb R^3$ space), so as to reduce the degrees of
freedom of equation and reduce the computational work of numerical simulation.
The idea of Kohn-Sham equation can be traced back to Thomas-Fermi model in 1927.  The model gave a primary description of the
electronic structure of atoms.  The strict theoretical analysis was described by
Hohenberg and Kohn in \cite{Hohenberg}, which proved the correctness and
feasibility of using single electron orbit to replace the wavefunctions
describing multiple electrons.

So far, lots of numerical methods for solving Kohn-Sham equation have been
developed. For instance, plane-wave method \cite{Canceschakir,liuchen} is the
most popular method in the computational quantum chemistry community. Since the
basis function is independent from the ionic position, plane-wave method has
advantage in calculating intermolecular force. Combined with the pseudopotential
method, plane-wave method plays an important role in the study of the ground and
excited states calculations, and geometry optimization of the electronic
structures. Although the plane-wave method is the most popular one in the
computational quantum chemistry community, it is inefficient in solving
non-periodic systems like molecules, nano-clusters, or materials systems with
defects, etc. Furthermore, the plane-wave method uses the global basis which
significantly affect the scalability on parallel computing platforms.  The
atomic-orbital-type basis sets \cite{HehreStewartPople,Jensen} are also widely
used for simulating materials systems such as molecules and clusters. However,
they are well-suited only for isolated systems with special boundary
conditions. It is difficult to develop a systematic basis-set for all materials
systems.  Thus over the past decades, more and more attentions are attracted to
develop efficient and scalable grid-based methods such as the finite element
method for electronic structure calculations.  The advantages of grid-based
methods include that it can use unstructured meshes and local basis sets, hence
owing high scalability on parallel computing platforms.  So far, the
applications of the finite element method in solving Kohn-Sham equation have
been studied systematically. We refer to
\cite{Baogang,BylaskaHostWeare,Castro,LinLuE,MasudKannan,Modine,Motamarri,PaskKleinSterneFong,PaskSterne,SchauerLinder,Skylaris,Suryanarayana}
and references therein for a comprehensive overview.

In this paper, we propose a nonnested augmented subspace method to solve the
Kohn-Sham equation based on the moving mesh adaptive technique and the augmented
subspace method.  Since the Coulomb potential and Hartree potential have strong
singularities, the adaptive mesh refinement is a competitive strategy to improve solving
efficiency.  Adaptive methods mainly include $h$-adaptive method, $p$-adaptive
method and $r$-adaptive method. The $h$-adaptive method uses \textit{a
  posteriori} error indicator for local refinement, $p$-adaptive method uses
higher-order polynomials on local mesh elements, $r$-adaptive method (moving
mesh method) uses control function or metric tensor (also known as metric
matrix) for mesh redistribution. In this paper, we adopt the $r$-adaptive method
to generate a series of nonnested adaptive finite element spaces. The
$r$-adaptive method uses the control function or metric matrix to guide the mesh
movement, which can improve the accuracy by moving the grid nodes to the area
with large errors.  This method can be traced back to Alexander's work in
\cite{Alexander}, and then Miller's work \cite{Miller1,Miller2} promoted the
development of $r$-adaptive method.  In 2001, Li and Tang etc
\cite{liruo1,liruo2} proposed the moving mesh method based on harmonic mapping.
Later, Di and Li etc \cite{liruo3} applied the moving mesh method based on
harmonic mapping to solve the incompressible Navier-Stokes equation.  More
results about moving mesh method can be found in \cite{di,Frey,xy,zhangw} and
the references therein.

Another technique adopted in this paper is the multilevel correction method
(augmented subspace method) \cite{JiaXieXu,LinXie,Xie_JCP,Xie_NonlinerEig}. The
traditional multilevel correction method can only be performed on the nested
multilevel space sequence. In this paper, we develop a nonnested multilevel
correction technique for Kohn-Sham equation on a nonnested mesh sequence
generated by the moving mesh technique. Different from the classical
self-consistent field iterative algorithm which requires to solve the Kohn- Sham
equation directly in each adaptive finite element space, our algorithm
transforms the Kohn-Sham equation into some linear boundary value problems of
the same scale on each level of the adaptive refinement mesh sequence, and then
the wavefunctions derived from the linear boundary value problems are corrected
by solving a small-scale Kohn-Sham equation defined in a low-dimensional
augmented subspace. Since the new algorithm avoids solving large-scale Kohm-Sham
equation directly, a significant improvement for the efficiency can be
obtained. In addition, the adaptive mesh is produced using the moving mesh
technique according to the singularity of the wavefunctions, which can
dramatically generate a high accuracy with less computational work. In our
algorithm, the approximate wavefunctions are used as the adaptive function, and
the modified Hessian matrix of the density function is used as the metric matrix to
redistribute the mesh. Through the moving mesh adaptive technique, the
redistributed mesh is almost optimal. By combining the moving mesh technique and
the nonnested augmented subspace method, the solving efficiency for Kohn-Sham
equation can be significantly improved.

The remainder of this paper is as follows. Section 2 introduces Kohn-Sham
equation and its finite element discretization method. Section 3 introduces the
nonnested augmented subspace method for Kohn-Sham equation by using the moving mesh
technique and the multilevel correction method. In Section 4, we propose the
convergence analysis and the estimate of computational work for the presented
algorithm. In Section 5, a number of numerical examples are demonstrated to
verify the efficiency and accuracy of the proposed method.

\section{Finite element method for Kohn-Sham equation}

To describe the Kohn-Sham equation and its finite element discretization, we
first introduce some notation. Following \cite{Adams}, we use $W^{s,p}(\Omega)$
to denote Sobolev spaces, and use $\|\cdot\|_{s,p,\Omega}$ and
$|\cdot|_{s,p,\Omega}$ to denote the associated norms and seminorms,
respectively. In case $p = 2$, we denote $H^s(\Omega) = W^{s,2}(\Omega)$ and
$H_0^1(\Omega) = \{v \in H^1(\Omega):v|_{\partial \Omega} = 0\}$, where
$v|_{\partial \Omega} = 0$ is in the sense of trace, and denote
$\|\cdot\|_{s,\Omega}=\|\cdot\|_{s,2,\Omega}$.  For convenience, the symbols
$ x_1 \lesssim y_1$, $x_2 \gtrsim y_2$ and $x_3 \approx y_3$ are used to
represent $x_1 \leq C_1y_1$, $x_2 \geq c_2y_2$, and
$c_3x_3 \leq y_3 \leq C_3 y_3$, respectively, where $C_1$, $c_2$, $c_3$, and
$C_3$ denote some mesh-independent constants.  In this paper, $\Omega$ is dropped from
the subscript of the norm for simplicity.

Let $\cH := (H_0^1(\Omega))^N$ be the Hilbert space with the inner product
\begin{eqnarray}
(\Phi,\Psi) = \sum_{i=1}^N\int_{\Omega}\phi_i\psi_idx,
\quad \forall\ \Phi = (\phi_1,\dots,\phi_N),
\ \ \Psi = (\psi_1,\cdots,\psi_N)\in \cH,
\end{eqnarray}
where $\Omega\subset \mathbb{R}^3$. For any
$\Psi\in\cH$ and a subdomain $\omega \subset \Omega$, we define
$\rho_{\Psi}=\sum_{i=1}^N|\psi_i|^2$ and
\begin{eqnarray*}
\|\Psi\|_{s,\omega} = \left( \sum_{i=1}^{N}
\|\psi_i\|_{s,\omega}^2\right)^{1/2}, \ \ s=0,1.
\end{eqnarray*}
Let $\cQ$ be a subspace of $\cH$ with orthonormality constraints:
\begin{eqnarray}
\cQ = \left\{  \Psi\in \cH: \Psi^T\Psi = I^{N\times N}    \right\},
\end{eqnarray}
where $\Phi^T\Psi = \big(\int_{\Omega}\phi_i\psi_jdx\big)_{i,j=1}^N\in \mathbb R^{N\times N}$.

We consider a molecular system consisting of $M$ nuclei with charges $\{Z_1$,
$\cdots$, $Z_M\}$ and locations $\{R_1$, $\cdots$, $R_M\}$, respectively, and
$N$ electrons in the non-relativistic and spin-unpolarized setting. The general
form of Kohn-Sham energy functional can be demonstrated as follows
\begin{eqnarray}
 E(\Psi) = \int_{\Omega} \left( \displaystyle\frac{1}{2}\sum_{i=1}^N|\nabla \psi_i |^2+V_{ext}(x)\rho_{\Psi}
 +E_{xc}(\rho_{\Psi})\right)dx
 +\displaystyle\frac{1}{2}D(\rho_{\Psi},\rho_{\Psi}),
\end{eqnarray}
for $\Psi = (\psi_1,\psi_2,\cdots,\psi_N)\in \cH$. Here, $V_{ext}$ is the
Coulomb potential defined by $V_{ext} = -\sum_{k=1}^{M}Z_k/|x-R_k|$,
$D(\rho_{\Phi},\rho_{\Phi})$ is the Hartree energy defined by
\begin{eqnarray}
 D(f,g) = \int_{\Omega}f(g*r^{-1})dx = \int_{\Omega}\int_{\Omega}f(x)g(y)\displaystyle\frac{1}{|x-y|}dxdy,
\end{eqnarray}
and $E_{xc}(t)$ is some real function over $[0,\infty)$ denoting the
exchange-correlation energy.

The ground state of the system is obtained by solving the minimization problem:
\begin{eqnarray}\label{minimization}
 \inf\{   E(\Psi):  \Psi \in \cQ     \},
\end{eqnarray}
and we refer to \cite{Canceschakir,ChenYang} for the existence of a minimizer
under some conditions.

The Euler-Lagrange equation corresponding to the minimization problem
(\ref{minimization}) is the well-known Kohn-Sham equation: Find
$(\Lambda, \Psi)\in \mathbb R^{N}\times \cH$ such that
\begin{equation}\label{Kohn_Sham_Equation}
\left\{
\begin{array}{l}
H_{\Psi}\psi_i =\lambda_i\psi_i\  \ {\rm in}\ \Omega, \quad i=1,2,\cdots,N,\\
\\
\displaystyle\int_\Omega \psi_i \psi_j dx =\delta_{ij},
\end{array}
\right.
\end{equation}
where $H_{\Psi}$ is the Kohn-Sham Hamiltonian operator defined by
\begin{eqnarray}
 H_{\Psi} = -\displaystyle\frac{1}{2}\Delta +V_{ext}
 +V_{Har}(\rho_{\Psi}) +V_{xc}(\rho_{\Psi})
\end{eqnarray}
with $V_{Har}(\rho_{\Psi})=r^{-1}*\rho_{\Psi}$, $V_{xc}(\rho_\Psi)=E'_{xc}(\rho_\Psi)$,
$\Lambda = (\lambda_1,\cdots,\lambda_N)$ and
$\lambda_i = (H_\Psi \psi_i,\psi_i)$.  The variational form of the Kohn-Sham
equation can be described as follows: Find
$(\Lambda, \Psi)\in \mathbb R^{N}\times \cH$ such that
\begin{equation}\label{Nonlinear_Eigenvalue_Problem2}
\left\{
\begin{array}{l}
(H_{\Psi}\psi_i,v) =\lambda_i(\psi_i,v),\quad  \forall v\in H_0^1(\Omega),\quad i=1,2,\cdots,N,\\
  \\
\displaystyle\int_\Omega \psi_i \psi_j dx = \delta_{ij}.
\end{array}
\right.
\end{equation}


Now, let us define the finite element discretization of
(\ref{Nonlinear_Eigenvalue_Problem2}).  First we generate a shape regular
decomposition $\mathcal{T}_h$ of the computing domain $\Omega$. Then, based on
the mesh $\cT_h$, we construct the linear finite element space denoted by
$S_h \subset H_0^1(\Omega)$. 
Define $V_h = (S_h)^N\subset \cH$.  Then the discrete form of
(\ref{Nonlinear_Eigenvalue_Problem2}) can be described as follows: Find
$(\bar\Lambda_h, \bar\Phi_h)\in \mathbb R^{N}\times V_h$ such that
\begin{equation}\label{Nonlinear_Eigenvalue_Problem2 fem}
\left\{
\begin{array}{l}
(H_{\bar\Psi_h}\bar\psi_{i,h},v) =\bar\lambda_{i,h}(\bar\psi_{i,h},v),
\quad\forall v\in S_h, \quad i=1,\cdots,N,\\
\\
\displaystyle\int_\Omega \bar\psi_{i,h} \bar\psi_{j,h} dx =\delta_{ij},
\end{array}
\right.
\end{equation}
with $\bar\lambda_{i,h} = (H_{\bar\Psi_h} \bar\psi_{i,h},\bar\psi_{i,h})$.

For simplicity and generality, some assumptions are given for the error analysis
of the proposed algorithm in this paper.  We would like to mention that the
following assumptions are satisfied by many practical physical models.  For
detail of the assumptions, please refer to
\cite{CancesChakirMaday,Canceschakir,ZhouAFEMKS,ChenZhou}.
\begin{itemize}
\item Assumption A: $|V_{xc}(t)|+|tV_{xc}'(t)|\in \mathcal P(p_1,(c_1,c_2))$ for
some $p_1\in [0,2]$, where $\mathcal P\big(p,(c_1,c_2)\big)$ denotes the
following function set:
\begin{eqnarray}\label{Function_Set_P} \mathcal P(p,(c_1,c_2))=\Big\{ f:\exists
a_1,a_2 \in \mathbb R \text{ such that } c_1t^p+a_1 \leq f(t) \leq c_2t^p +a_2,\
\forall t \geq 0 \Big\}
\end{eqnarray} with $c_1 \in \mathbb R$ and $c_2$, $p \in [0,\infty)$.
\item Assumption B: There exists a constant $\alpha \in (0,1] $ such that
$|V_{xc}'(t)|+|tV_{xc}''(t)|\lesssim 1+t^{\alpha-1}$.

\item Assumption C: Let $(\Lambda, \Psi)$ be a solution of
(\ref{Nonlinear_Eigenvalue_Problem2}).  When the mesh size $h$ is small enough,
the discrete solution $(\bar\Lambda_{h},\bar\Psi_h)\in \mathbb R^N \times V_h$
satisfies the following estimate
\begin{eqnarray} \|\Psi-\bar\Psi_h\|_1 &\lesssim& \delta_{h}(\Psi), \\
\|\Psi-\bar\Psi_h\|_0+|\Lambda-\bar\Lambda_h|&\lesssim&
r(V_h)\|\Psi-\bar\Psi_h\|_1,
\end{eqnarray} where $r(V_h) \ll 1$ and $$\delta_{h}(\Psi):=\inf\limits_{\Phi
\in V_h}\|\Psi-\Phi\|_1.$$
\end{itemize}

\section{Nonnested augmented subspace method for Kohn-Sham equation}

In this section, we design the nonnested augmented subspace method for Kohn-Sham
equation based on the moving mesh adaptive technique and the augmented subspace
method. In the first two subsections, we introduce some computing techniques
including the solving process for exchange-correlation potential and Hartree
potential, and the detailed process of moving mesh technique. Then next, we
combine these techniques to generate the nonnested augmented subspace method for
Kohn-Sham equation.

\subsection{The calculation of exchange-correlation potential and Hartree potential}

The Kohn-Sham equation contains two nonlinear terms: the Hartree potential and
the exchange-correlation potential.  In this subsection, we introduce the
detailed form for exchange-correlation potential and the solving process for
Hartree potential.

Since there is no exact expression for the exchange-correlation potential, we
use the local density approximation (LDA) in our numerical simulations.  The
exchange-correlation potential can be treated as the variational derivative of
the exchange-correlation energy to the density function
$V_{xc}=\delta E_{xc}^{LDA}/\delta\rho$. The exchange-correlation
energy $E_{xc}^{LDA}$ is divided into two parts: exchange energy $E_{x}^{LDA}$ and correlation
energy $E_{c}^{LDA}$, that is $E_{xc}^{LDA}=E_{x}^{LDA}+E_{c}^{LDA}$.  For the exchange energy, we use
the following approximate form given by Kohn and Sham in \cite{Kohn}:
\begin{eqnarray}
E_{x}^{LDA}(\rho) = -\displaystyle\frac{3}{4}\left(\displaystyle\frac{3}{\pi}\right)^{1/3}\int_\Omega \rho(x)^{4/3}d\Omega,
\end{eqnarray}
with the exchange-correlation potential
\begin{eqnarray}\label{expot}
V_{x}^{LDA}(\rho) = -\left(\displaystyle\frac{3\rho}{\pi}\right)^{1/3}.
\end{eqnarray}

For the relatively complex correlation energy and correlation potential, we use
the expression proposed by Perdew and Zunger in \cite{perdew}.  The correlation
energy per electron is:
\begin{equation*}
\varepsilon_{c}^{LDA}=\left\{
\begin{array}{l}
  \ A  \text{ln} r_s+B+C r_s\text{ln} r_s +D r_s,\ \ \ \ \text{if}\ \  r_s<1,\\
  \\
\displaystyle\frac{ \gamma}{(1+\beta_1\sqrt{r_s}+\beta_2 r_s)},\ \ \ \text{if}\ \  r_s\geq 1.
\end{array}
\right.
\end{equation*}
The corresponding correlation potential is
\begin{equation}\label{copot}
V_{c}^{LDA}(\rho)=\left\{
\begin{array}{l}
  \ A  \text{ln} r_s +(B-\displaystyle\frac{1}{3}A)+\displaystyle\frac{2}{3}C r_s\text{ln} r_s +\displaystyle\frac{1}{3}(2D-C) r_s,\ \ \ \ \text{if}\ \  r_s<1,\\
  \\
  \displaystyle\varepsilon_{c}^{LDA}\displaystyle\frac{1+\displaystyle\frac{7}{6}\beta_1\sqrt{r_s}+\displaystyle\frac{4}{3}\beta_2r_s}{(1+\beta_1\sqrt{r_s}+\beta_2r_s)},\ \ \ \text{if}\ \  r_s\geq 1
\end{array}
\right.
\end{equation}
with
$r_s=(3/(4\pi\rho))^{1/3}, A=0.0311, B=-0.048, \gamma=-0.1423,
\beta_1=1.0529, \beta_2=0.3334, C=0.0020, D=-0.0116$.

The ground state energy of the molecular system can be calculated through:
\begin{eqnarray}
  E=\sum_{i=1}^N\lambda_i -\int_\Omega \left(\displaystyle\frac{1}{2}V_{Hartree}+V_{xc}\right)\rho(x)d\Omega +E_{xc}+E_{nn}.
\end{eqnarray}
The last term $E_{nn}$ accounts for the interactions between the nuclei:
\begin{eqnarray}
  E_{nn}=\sum_{1\leq k<j\leq M}\displaystyle\frac{Z_jZ_k}{|x_j-x_k|}.
\end{eqnarray}

Next, we introduce the solving process for Hartree potential $V_{Har}$, which is computed by solving the
following Poisson equation:
\begin{equation}
\left\{
\begin{array}{l}
  -\Delta V_{Har}=4\pi \rho, \ \ \text{in}\ \Omega,\\
\\
V_{Har}=w_D, \ \ \ \ \ \ \ \ \text{on}\ \partial\Omega.
\end{array}
\right.
\end{equation}
Because the Hartree potential decays with the rate $N/r$, where $N$ is the
electron number, the simple use of zero Dirichlet boundary condition will
introduce large truncation error on the boundary. To give the evaluation of the
Hartree potential on the boundary, a multipole expansion approximation is
employed for the boundary values. In our numerical simulations, the following
approximation is used:
\begin{equation}\label{bd}
  w_D=w|_{\partial\Omega}\approx \displaystyle\frac{1}{|\bf x-x''|}\int_{\Omega}\rho({\bf x'})d{\bf x'} +\sum_{i=1,2,3}p_i\cdot \displaystyle\frac{x^i-x''^{,i}}{|\bf x-x''|^3} +\sum_{i,j=1,2,3}q_{i,j}\displaystyle\frac{3(x^i-x''^{,i})(x^j-x''^{,j})-\delta_{i,j}|\bf x-x''|^2}{|\bf x-x''|^5},
\end{equation}
where
\begin{eqnarray*}
  p_i= \int_{\Omega}\rho({\bf x'})(x'^i-x''^{,i})d{\bf x'}, \ \ \ \ \ \ q_{i,j}= \int_{\Omega}\rho({\bf x'})(x'^i-x''^{,i})(x'^j-x''^{,j})d{\bf x'}.
\end{eqnarray*}
In the above expressions, ${\bf x''}$ is chosen as ${\bf x''}= \int_{\Omega}{\bf x}\rho({\bf
    x})d{\bf x}\big/\int_{\Omega}\rho({\bf x})d{\bf x}$.

\subsection{Moving mesh adaptive technique}

The adaptive methods mainly include $h$-adaptive method, $p$-adaptive method and
$r$-adaptive method.  The $h$-adaptive method uses \textit{a posteriori} error
indicator for local refinement, $p$-adaptive method uses higher-order
polynomials on local mesh elements, $r$-adaptive method uses control function or
metric tensor (also known as metric matrix) for mesh redistribution.  In this
subsection, we introduce the detailed process for the $r$-adaptive method which
is used for mesh adaptive refinement in this paper.

The moving mesh adaptive ($r$-adaptive) software packages mostly control mesh
movement based on the control function or metric tensor.  Some well-known finite
element software packages such as AFEPack \cite{liruo3,liruo1,liruo2} are based
on the control function.  Other software packages such as BAMG \cite{Hecht2},
Mshmet and Mmg \cite{Dapogny,Frey} are based on metric tensor.  Essentially, the
control function and the measurement matrix are consistent.  From mathematical
point of view, they differ only by a constant factor.  From the way of
controlling the movement of the grid points, the control function is based on
the error equal distribution principle, and the measurement matrix is based on
the unit volume principle (BAMG) or the unit edge principle (mshmet and Mmg).
Here, we use the three-dimensional anisotropic adaptive software Mshmet and
Mmg3d (three-dimensional module in Mmg) based on the $M$-unit side length
principle, in which the measurement used is the correction of the Hessian matrix
$H_u(z)$ of the given adaptive function $u$ at the grid node $z$. The detailed
form is as follows:
\begin{eqnarray}\label{m1}
  \mathcal M(z)= R\widetilde \Lambda R^{-1},
\end{eqnarray}
where $z$ denotes the location of the grid node, $R$ denotes an orthogonal matrix
composed of the orthogonal eigenfunctions of the Hessian matrix $H_u(z)$,
$\widetilde \Lambda$ is defined by

\begin{equation}\label{st2}
\widetilde\Lambda:=\left(
\begin{array}{cccc}
\widetilde \lambda_1 &  &  & \\
&  \widetilde \lambda_2 &  & \\
&   &  \ddots & \\
&   &   &  \widetilde \lambda_n
\end{array}
\right).
\end{equation}
To guarantee the positive definite of the measurement matrix,
$\widetilde \lambda_i$ is set to be
\begin{eqnarray}\label{m2}
  \widetilde \lambda_i = \min \left( \max \Big( \displaystyle\frac{C_d|\lambda_i|}{\varepsilon},  \displaystyle\frac{1}{h_{\max}^2} \Big),\displaystyle\frac{1}{h_{\min}^2}\right),
\end{eqnarray}
where $\varepsilon$ is used to control the $P_1$-interpolation error in the
sense of $L^\infty$-norm, $\lambda_i$ denotes the eigenvalue of the Hessian matrix $H_u(z)$, $h_{max}$ and $h_{min}$ represent the upper limit of
the longest edge scale and the lower limit of the shortest edge scale,
respectively.  In addition, the coefficient $C_d$ is selected based on the
following $L^\infty-$error estimate (see \cite{Alauzet}):
\begin{lemma}
  Let $K$ be a mesh element (d-dimensional simplex) of mesh $\mathcal T_h$, $E_K$
  is a set of all edges on $K$, and $u$ is a quadratic differentiable
  function. Then, we have the following $L^\infty-$interpolation error estimate
\begin{eqnarray}\label{est1}
  \|u-\Pi_hu\|_{\infty,K} \leq C_d \max_{{\bf y}\in K}\max_{{\bf e}\in E_k} {\bf e}^T|H_u({\bf y})|{\bf e}, \ \ \ \forall K\in \mathcal T_h,
\end{eqnarray}
where
\begin{eqnarray}
  C_d=\displaystyle\frac{1}{2}\left(\displaystyle\frac{d}{d+1}\right)^2, \ \ |H_u|:=R |\Lambda| R^{-1},
\end{eqnarray}
with $|\Lambda|:=diag \{ |\lambda_1|,|\lambda_2|,\cdots,|\lambda_n| \}$.
\end{lemma}

The computation of the right term of (\ref{est1}) is nontrivial. Thus, we first
assume that there exists a measurement $\overline{\mathcal M}(K)$, such that
\begin{eqnarray}
  \max_{{\bf y}\in K} {\bf e}^T|H_u({\bf y})|{\bf e} \leq {\bf e}^T \overline{\mathcal M}(K){\bf e},\quad \forall {\bf e}\in E_k.
\end{eqnarray}
To derive $\|u-\Pi_hu\|_{\infty,K} \leq \varepsilon$, it is sufficient to
require that
\begin{eqnarray}\label{lonp}
  C_d{\bf e}^T \overline{\mathcal M}(K){\bf e}=\varepsilon,\quad \forall {\bf e}\in E_k, \ \ \text{or} \ \ C_d\max_{{\bf e}\in E_k}{\bf e}^T \overline{\mathcal M}(K){\bf e}=\varepsilon.
\end{eqnarray}
By defining the measurement $\mathcal M := (C_d/\varepsilon) \overline{\mathcal M}$, and based on
(\ref{lonp}), we know the measurement satisfies the following $\mathcal M-$unit
edge length principle
\begin{eqnarray}\label{lonp2}
{\bf e}^T \mathcal M(K){\bf e}=1,\quad \forall {\bf e}\in E_k.
\end{eqnarray}
That is, the length of side ${\bf e}$ of element $K$ under $\mathcal M$ is 1,
which is denoted by
\begin{eqnarray}\label{lonp3}
\ell_{ \mathcal M(K)}({\bf e})=1,\quad \forall {\bf e}\in E_k.
\end{eqnarray}

In the actual calculation, $\mathcal M$ can be determined by
(\ref{m1})$-$(\ref{m2}).  In fact, in the software package Mshmet, we only need
to adjust the parameters err$(\varepsilon)$, hmin$(h_{min})$, hmax$(h_{max})$ to
control the edge length of mesh elements. For more detailed introduction to the
metric based adaptive mesh using the $M-$unit edge length principle, please
refer to \cite{Alauzet,Frey}. For more discussion on metric matrices and control
functions, please refer to \cite{di,xy,zhangw}.

Next, we introduce a fast interpolation algorithm between two nonnested meshes.
Let $\Omega_0$ and $\Omega_1$ be two different computing domains,
$\mathcal T_h^{(0)}=\cup_iK_i^{(0)}$ and $\mathcal T_h^{(1)}=\cup_iK_i^{(1)}$ be
the corresponding mesh decompositions on these two domains, respectively. The
two decompositions are nonnested.  Then, we briefly introduce the fast
interpolation algorithm proposed by the FreeFEM team to deal with nonnested
meshes \cite{Hecht}.

Let
$V_h^{(\ell)}:=\big\{v\in C^0(\mathcal T_h^{(\ell)}) \big|\
v|_{{K_i}^{(\ell)}}\in P_1\big\}, \ \ell=0,1,$ be the finite element space
defined on $\mathcal T_h^{(0)}$ and $\mathcal T_h^{(1)}$. For $f\in V_h^{(0)}$,
the problem is to find $g\in V_h^{(1)}$ such that:
\begin{eqnarray}
g(q)=f(q),\quad \forall \text{ vertex } q \text{ of } \mathcal T_h^{(1)}.
\end{eqnarray}
A fast interpolation algorithm is proposed in \cite{Hecht} which is of
complexity $\mathcal O(N_1 {\rm{ln}} N_0)$, where $N_0$ is the number of
vertices of $\mathcal T_h^{(0)}$, $N_1$ is the number of vertices of
$\mathcal T_h^{(1)}$.  The detailed process is presented in Algorithm
\ref{interp}.

\begin{algorithm}[H]
  \KwData{First a quadtree is built containing all the vertices of the
    mesh $\mathcal T_h^{(0)}$ such that in each terminal element there
    are at least one, and at most $2^d$ vertices.}

  \KwResult{$g(q^1)$}

  \For{each vertex $q^1$ in $\mathcal T_h^{(1)}$}{
    Find the terminal cell of the quadtree containing $q^1$\;
    Find the the nearest vertex $q_j^0$ to $q^1$ in that cell\;
    Choose one triangle $K_i^{(0)}\in \mathcal T_h^{(0)}$ which has $q_j^0$ for vertex\;
    Compute the barycentric coordinates $\{\lambda_j\}_{j=1}^{d+1}$ of $q^1$ in $K_i^{(0)}$\;
    \While{Not all barycentric coordinates are positive}{

      \uIf{One barycentric coordinate $\lambda_i$ is negative}{
        replace $K_i^{(0)}$ by the adjacent triangle opposite $q_i^0$
      }
      \ElseIf{two barycentric coordinates are negative}{
        take one of the two randomly and replace $K_i^{(0)}$ by the adjacent triangle as above
      }
      Compute the barycentric coordinates $\{\lambda_j\}_{j=1}^{d+1}$ of $q^1$ in $K_i^{(0)}$\;
    }
    Calculate $g(q^1)$ on $K_i^{(0)}$ by linear interpolation of $f$: $g(q^1)=\sum_{j=1}^{d+1} \lambda_j f(q_j^{(0)})$\;
  }
  \caption{Fast interpolation algorithm in FreeFEM}
  \label{interp}
\end{algorithm}


\subsection{Nonnested augmented subspace method for Kohn-Sham equation}

In this subsection, we introduce the nonnested augmented subspace
algorithm for Kohn-Sham equation.  To simplify the description, we define a
nonlinear operator which is composed of the Kohn-Sham Hamiltonian operator and a
positive constant:
\begin{eqnarray}
\widetilde{H}_\Psi = -\displaystyle\frac{1}{2}\Delta +V_{ext}+\mu+ V_{Har}(\rho_\Psi)+ V_{xc}(\rho_\Psi),\quad\forall \Psi\in V,
\end{eqnarray}
where $\mu$ is a positive constant that can be used to guarantee the coercive
property.

The essence of the nonnested augmented method is to transform the solution of
the Kohn-Sham equation on the fine mesh into solving the linear boundary value
problems of the same scale and solving a small-scale Kohn-Sham equation in a
low-dimensional subspace.  In order to describe the nonnested augmented subspace
algorithm, we generate a coarse mesh $\mathcal T_H$ and corresponding linear
finite element space $S_H$.  The traditional multilevel correction method
requires the multilevel finite element space sequence satisfies the nested
relationship
\begin{eqnarray}
  \mathcal T_H \subset \mathcal T_{h_1} \subset \cdots \subset \mathcal T_{h_N},\text{ and } S_H \subset  S_{h_1} \subset \cdots \subset S_{h_N}.
\end{eqnarray}
The nonnested augmented subspace method presented in this paper does not require
the mesh and finite element space to meet the above nested property, which
guarantees that the augmented subspace method can be used in moving mesh
sequence. In Algorithm \ref{One_Correction_Step}, the flowchart of the moving
mesh adaptive algorithm for the Kohn-Sham equation is given.

\begin{algorithm}[H]\label{One_Correction_Step}
  \KwData{A given electronic structure, and the initial guess of the ground state}
  \KwResult{Ground state of the given electronic structure}

  Give $K_{asm}, K_{max}\in\mathbb{N}$ with $K_{asm} < K_{max}$, $tol\in\mathbb{R}^+$, and let $k = 0$\;
  \Repeat{$k > K_{asm}$ or $\Delta E/E < tol$}{
    Using the standard self-consistent field iteration method to solve the problem: Find
    $(\Lambda_{h_k},\Psi_{h_k})\in \mathbb R\times V_{h_k}$, such
    that 
    \begin{equation}\label{ksalg1}
      \left\{
        \begin{array}{rcl}
          ( H_{\Psi_{h_k}}\psi_{i,h_k},v_{h_k})  &=&\lambda_{i,h_k}(\psi_{i,h_k},v_{h_k}),
                                                     \quad\forall v_{h_k}\in S_{h_k},\ i=1,\cdots,N,\\
          \int_\Omega \psi_{i,h_k} \psi_{j,h_k} dx &=&\delta_{ij}.
        \end{array}
      \right.
    \end{equation}

    Move the mesh grids according to the measurement (\ref{m1}) by using $\sqrt{\rho_{h_{k}}}$ as the adaptive function\;
    Update the mesh from $\mathcal T_{h_{k}}$ to $\mathcal T_{h_{k+1}}$ and interpolate the wavefunctions into the new mesh\;
    $k = k + 1$\;
  }
  \While{$k < K_{max}$ and $\Delta E/E > tol$}{
    Using Algorithm \ref{MMAA} to solve the above problem (\ref{ksalg1})\;
    Move the mesh grids according to the measurement (\ref{m1}) by using $\sqrt{\rho_{h_{k}}}$ as the adaptive function\;
    Update the mesh from $\mathcal T_{h_{k}}$ to $\mathcal T_{h_{k+1}}$ and interpolate the wavefunctions into the new mesh\;
    $k = k + 1$\;

  }

  \caption{Moving mesh adaptive algorithm for Kohn-Sham equation}
\end{algorithm}

Next, we introduce the an augmented subspace method to solve the Kohn-Sham equation
(\ref{ksalg1}), which transforms the large-scale nonlinear eigenvalue problem
into some linear boundary value problems of the same scale and small-scale
Kohn-Sham equations defined in a low-dimensional augmented subspace. Through the
augmented subspace method, the total computational work is asymptotically
optimal since the dimension of the augmented subspace is small and fixed. The
detailed process is described in Algorithm \ref{MMAA}.

\begin{algorithm}[!htbp]
  \caption{Augmented subspace method for Kohn-Sham equation}\label{MMAA}
  \KwData{The initial value $(\Lambda_{h_k}^{(0)},\Psi_{h_k}^{(0)})$}
  \KwResult{The ground state $(\Lambda_{h_k}^{(\ell)},\Psi_{h_k}^{(\ell)})$}
  Give $K_{max}\in\mathbb{N}$, and $tol\in\mathbb{R}^+$, and let $\ell = 0$\;
  \Repeat{$\|\rho_{h_k}^{(\ell)}-\rho_{h_k}^{(\ell - 1)}\|_0 < tol$ or $k > K_{max}$}{
    Solve the following $N$ problems: Find $\widehat{\psi}_{i,h_k}^{(\ell+1)} \in S_{h_k}$ such that
    \begin{eqnarray}\label{Aux_Linear_Problem}
      \begin{array}{rcr}
        \left(\widetilde H_{\Psi_{h_k}^{(\ell)}}\widehat{\psi}_{i,h_k}^{(\ell+1)},v_{h_k}\right) = \left(\lambda_{i,h_k}^{(\ell)}+\mu\right)(\psi_{i,h_k}^{(\ell)},v_{h_k}),
        \ \ \ \forall v_{h_k} \in S_{h_k}, \ i=1,\cdots,N.
      \end{array}
    \end{eqnarray}

    Solve the following small-scale Kohn-Sham equation in an augmented subspace
    $S_{H,h_k}=S_H+{\rm span}\{\widehat{\Psi}_{h_k}^{(\ell+1)}\}$: Find
    $(\lambda_{i,h_k}^{(\ell+1)},\psi_{i,h_k}^{(\ell+1)}) \in \mathbb R\times
    S_{H,h_k}$ such that
    \begin{eqnarray}\label{Nonlinear_Eig_Hh}
      \begin{array}{rcr}
        && \left(H_{\Psi_{h_k}^{(\ell+1)}}\psi_{i,h_k}^{(\ell+1)},v_{h_k}\right) = \lambda_{i,h_k}^{(\ell+1)}(\psi_{i,h_k}^{(\ell+1)},v_{h_k}),\ \ \ \forall v_{h_k} \in S_{h_k},\  i=1,\cdots,N,
      \end{array}
    \end{eqnarray}
    to get the new approximate eigenpair $(\Lambda_{h_k}^{(\ell+1)}, \Psi_{h_k}^{(\ell+1)})$\;
    Update the density function $\rho_{h_k}^{(\ell+1)}=\sum_{i=1}^N\psi_{i,h_k}^{(\ell+1)}$\;
    Let $\ell = \ell + 1$\;
  }
\end{algorithm}

From Algorithm \ref{MMAA}, we can find that the Kohn-Sham equation
(\ref{Nonlinear_Eig_Hh}) is defined in a low-dimensional space $S_{H,h_k}$, thus
a small-scale linear eigenvalue problem is needed to be solved in each nonlinear
iteration step, which requires little computational work.  But the augmented
subspace $S_{H,h_k}$ includes the basis functions
$\widehat{\Psi}_{h_k}^{(\ell+1)}$ derived from the fine mesh $\mathcal T_{h_k}$,
so the matrices assembling should be performed on the fine mesh to keep the
accuracy.

Next, we introduce the detailed process for solving the small-scale Kohn-Sham
equation (\ref{Nonlinear_Eig_Hh}) in $S_{H,{h_k}}$.  For simplicity, we use $h$,
$V_h$, $\widehat\psi_{i,h}$ to denote $h_k$, $V_{h_k}$,
$\widehat\psi_{i,h_k}^{(\ell+1)}$, respectively.  Define $N_H:=\text{dim} S_H$
and $N_k:=\text{dim} S_{h_k}$. Let $\{\psi_{j,H}\}_{j=1}^{N_H}$ be the basis
functions for $S_H$.  When solving (\ref{Nonlinear_Eig_Hh}) by iteration method,
a linear eigenvalue problem as follows is required to be solved in each
iteration step:
\begin{eqnarray}\label{st1}
F_AU_H=\lambda_h F_BU_H,
\end{eqnarray}
where
\begin{equation}\label{st2}
F_A:=\left(
\begin{array}{cc}
A_H & {\bf b}_H\\
{\bf b}_H^T&{\bf \beta}
\end{array}
\right), \ \
F_B:=\left(
\begin{array}{cc}
B_H & {\bf c}_H\\
{\bf c}_H^T&{\bf \gamma}
\end{array}
\right), \ \
U_H=\left(
\begin{array}{c}
\mathbf{u}_{H} \\
{\bf \alpha}
\end{array}
\right),
\end{equation}
with $A_H, B_H\in \mathbb R^{N_H\times N_H}$,
${\bf b}_H, {\bf c}_H \in \mathbb R^{N_H\times N}$,
$\beta, \gamma \in \mathbb R^{N\times N}$.  Here,
$U_H=(\mathbf{u}_{H}, \alpha)^{\rm T}$ is the eigenvector to be solved in each
iteration step, and $\text{dim}U_H=\text{dim}V_H+N$, $N$ is the number of
eigenpair to be solved. The mass matrix
$B_H=\big( (\phi_{i,H},\phi_{j,H})_{\mathcal T_H} \big)_{1\leq i,j\leq N_H}$
remains unchanged during the nonlinear iteration, which can be assembled once at
the beginning of Algorithm \ref{MMAA}.  Hereafter, $(\cdot,\cdot)_{\mathcal T}$
denotes the $L^2$-inner product on the mesh $\mathcal T$.

For the matrices
$c_H=\big( (\widehat\psi_{j,h},\psi_{i,H})_{\mathcal T_h} \big)_{1\leq i\leq
  N_H, 1\leq j\leq N}$ and
$\gamma=\big( (\widehat\psi_{j,h},\widehat\psi_{i,h})_{\mathcal T_h}
\big)_{1\leq i,j\leq N}$, the involved basis functions
$\{\widehat\psi_{j,h}\}_{1\leq j\leq N}$ are defined on the fine mesh
$\mathcal T_h$.  Thus, to keep the accuracy, these matrices assembling should be
performed on the fine mesh $\mathcal T_h$.  As we can see, the basis functions
$\{\widehat\psi_{j,h}\}_{1\leq j\leq N}$ remain unchanged during the nonlinear
iteration step, thus the matrices $c_H$ and $\gamma$ also can be assembled once
at the beginning of Algorithm \ref{MMAA}.

It is noted that, based on the above discussion, the mass matrix $F_B$ only need
to be assembled once at the beginning of the nonlinear iteration in Algorithm
\ref{MMAA}. No updates are required during the iteration.

Next, we discuss the stiffness matrix $F_A$. Let us divide the stiffness matrix
into linear and nonlinear parts.  The linear part $A_h^L$ is defined by:
\begin{eqnarray}\label{st3} (A_h^L)_{i,j} = \displaystyle\frac{1}{2}(\nabla \psi_{i,h},\nabla
\psi_{j,h})_{\mathcal T_h} +(V_{ext} \psi_{i,h}, \psi_{j,h})_{\mathcal T_h}+(\mu
\psi_{i,h}, \psi_{j,h})_{\mathcal T_h}.
\end{eqnarray}
The nonlinear part $A_h^{NL}$ consisting of the Hartree potential and
exchange-correlation potential is defined by:
\begin{eqnarray}\label{st4} (A_h^{NL})_{i,j} = ((V_{Har}+V_{xc}) \psi_{i,h},
\psi_{j,h})_{\mathcal T_h},
\end{eqnarray}
where $V_{Har}$ and $V_{xc}$ need to be updated in each nonlinear iteration
because they depends on the density function that will be updated after each
iteration.  Let us define $A_h:=A_h^{L}+A_h^{NL}$.  Using the interpolation
algorithm defined in Algorithm \ref{interp}. The matrix $A_H$ involved in $F_A$
can be calculated in the following way
\begin{eqnarray}\label{st5}
  A_H = (I_{H}^{h})^{\rm T}A_hI_{H}^{h},
\end{eqnarray}
where $I_H^h$ denotes the interpolate operator from $S_H$ to $S_{h}$ based on
Algorithm \ref{interp}.

The remaining parts $b_H$ and $\beta$ can be calculated by:
\begin{eqnarray}\label{st6} b_H = (I_{H}^{h})^{\rm T}A_h\widehat \Psi_h
\end{eqnarray} and
\begin{eqnarray}\label{st7} \beta = \widehat\Psi_h^{\rm T}A_h\widehat \Psi_h.
\end{eqnarray}

\begin{remark}
  Based on the assembling process (\ref{st1})$-$(\ref{st7}), the main
  computational work is spent on (\ref{st4}) in each nonlinear
  iteration. Fortunately, the matrix assembling can be performed in parallel
  since there is no date transfer, and meanwhile the software package used in
  this paper has an excellent parallel ability \cite{Jolivet,Moulin}, thus we
  assemble the matrix using parallel computing technique.  
\end{remark}

\subsection{Convergence analysis and computational work estimate}

In this subsection, we give the convergence analysis and computational
complexity estimate for Algorithms \ref{One_Correction_Step} and \ref{MMAA}.
First, we can obtain the following theorem to guarantee the well-posedness of
the linear boundary value problems.
\begin{theorem}\label{coercive lemma}
  With sufficiently small mesh size, there exists a positive constant $\mu$ such
  that the following coercive property holds true
\begin{eqnarray}\label{coer}
  a_\mu (\phi,\phi) +(V_{Har}(\rho_{\Psi_h})\phi,\phi) +(V_{xc}(\rho_{\Psi_h})\phi,\phi)   \gtrsim\|\nabla \phi\|_0^2,\quad\forall \phi\in H_0^1(\Omega),
\end{eqnarray}
where
\begin{eqnarray}
  a_\mu (\phi,\phi)=\displaystyle\frac{1}{2}(\nabla \phi,\nabla \phi)+(V_{ext}\phi,\phi) + \mu(\phi,\phi).
\end{eqnarray}
\end{theorem}

\begin{proof}
  For the Coulomb potential $V_{ext} = -\sum_{k=1}^{M}Z_k/|x-R_k|$, using
  the following uncertainty principle lemma \cite{Reed}:
\begin{eqnarray}\label{uncertainlemma}
  \int_{\Omega}\displaystyle\frac{w^2(x)}{|x|^2}dx\leq 4\int_{\Omega}|\nabla w|^2dx,\quad\forall w\in H_0^1(\Omega),
\end{eqnarray}
we obtain
\begin{eqnarray}\label{est_col}
(V_{ext}\phi,\phi)&\leq& \left(\int_{\Omega}\Big(\sum\limits_{k=1}^{M}\displaystyle\frac{Z_k\phi }{|x-R_k|}\Big)^2dx\right)^{1/2}\|\phi\|_0 \nonumber\\
&\leq& \left(\int_{\Omega}\Big(\sum\limits_{k=1}^{M}Z_K^2\Big)
\Big(\sum\limits_{k=1}^{M}\displaystyle\frac{\phi^2 }{|x-R_k|^2}\Big)dx\right)^{1/2}\|\phi\|_0\nonumber\\
&\leq& 2\sqrt{M}\left(\sum\limits_{k=1}^{M}Z_K^2\right)^{1/2}\|\nabla \phi\|_0\|\phi\|_0 \nonumber\\
&\leq& \displaystyle\frac{1}{8}\|\nabla \phi\|_0^2 +8M\left(\sum\limits_{k=1}^{M}Z_K^2\right)\|\phi\|_0^2.
\end{eqnarray}
For the Hartree potential, using (\ref{uncertainlemma}) and H\"{o}lder inequality, we have
\begin{eqnarray}\label{est_har}
(V_{Har}(\rho_{\Psi_h})\psi,\psi)&\leq& \|r^{-1}\ast \rho_{\Psi_h}\|_{0,\infty}\|\phi\|_0\|\phi\|_0 \nonumber\\
&\leq& 2\|\nabla \Psi_h\|_0\|\Psi_h\|_0\|\nabla \phi\|_0\|\phi\|_0\nonumber\\
&\leq& \displaystyle\frac{1}{8}\|\nabla \phi\|_0^2 +8\|\nabla \Psi_h\|_0^2\|\Phi_h\|_0^2\|\phi\|_0^2.
\end{eqnarray}
For the exchange-correlation potential, we have
\begin{eqnarray}\label{mva}
(  V_{xc}(\rho_{\Psi_{h}})\phi,\phi  ) &=& (  ( V_{xc}(\rho_{\Psi_{h}}) - V_{xc}(0) )\phi,\phi  ) +  ( V_{xc}(0)\phi,\phi  ).
\end{eqnarray}
From Assumption A, we have $|V_{xc}(0)|\leq a_2$. For the first part of
(\ref{mva}), the following estimate holds
\begin{eqnarray}\label{mvb}
  V_{xc}(\rho_{\Psi_{h}}) - V_{xc}(0) &=& 2\sum_{i=1}^{N}V_{xc}'(\rho_{\Psi_{\varepsilon}})\psi_{\varepsilon,i}\psi_i
\end{eqnarray}
with $\psi_{\varepsilon,i}=\theta\psi_{i,h}+(1-\theta)0=\theta\psi_{i,h}$ and
$\theta\in [0,1]$. Using Assumption B, we can further derive
\begin{eqnarray}\label{mvc}
V_{xc}(\rho_{\Psi_h})-V_{xc}(0)&\leq& 2\left(1+\left(\sum_{i=1}^{N}\psi_{\varepsilon,i}^2\right)^{\alpha-1}\right)
\sum_{i=1}^{N}\psi_{\varepsilon,i}\psi_{i,h}\nonumber\\
&=&2\left(1+\left(\sum_{i=1}^{N}(\theta\psi_{i,h})^2\right)^{\alpha-1}\right)\sum_{i=1}^{N}\theta\psi_{i,h}^2
\leq 2\rho_{\Psi_h}+2\rho_{\Psi_h}^\alpha.
\end{eqnarray}
Inserting (\ref{mvc}) into (\ref{mva}) and using Holder inequality, we can
derive
\begin{equation}\label{mvd}
  \begin{array}{lll}
    (  V_{xc}(\rho_{\Psi_{h}})\phi,\phi  ) &\leq& 2(\rho_{\Psi_h}\phi,\phi) +2(\rho_{\Psi_h}^\alpha\phi,\phi) +a_2\|\phi\|_0^2  \\
    &&\\
                                           &\leq& 2\|\rho_{\Psi_h}\|_{0,3}\|\phi\|_0\|\phi\|_{0,6} +2\|\rho_{\Psi_h}^\alpha\|_{0,3/\alpha}\|\phi\|_0\|\phi\|_{0,6/(3-2\alpha)}  +a_2\|\phi\|_0^2 \\
    &&\\
                                           &\leq& 2C_{em}(\|\nabla\Psi_h\|_{0}^2+\|\nabla\Psi_h\|_0^{2\alpha})\|\phi\|_0\|\nabla\phi\|_{0} +a_2\|\phi\|_0^2, \\
                                           &&\\
&\leq& \displaystyle\frac{1}{8}\|\nabla\phi\|_{0}^2 + 8C_{em}^2(\|\nabla\Psi_h\|_{0}^2+\|\nabla\Psi_h\|_0^{2\alpha})^2\|\phi\|_0^2 +a_2\|\phi\|_0^2,
       \end{array}
\end{equation}
where $C_{em}$ is the mesh-independent constant involved in embedding theorem
$H_0^1(\Omega)\hookrightarrow L^6(\Omega)$, when $\Omega\subset \mathbb R^3$.

Combining (\ref{est_col}), (\ref{est_har}) and (\ref{mvd}), the following
inequality holds
\begin{equation}
  \begin{array}{l}
    a_\mu (\phi,\phi) +(V_{Har}(\rho_{\Psi_h})\phi,\phi) +(V_{xc}(\rho_{\Psi_h})\phi,\phi)\\\hskip 5em \geq \displaystyle\frac{1}{2}\|\nabla\phi\|_0^2 +\displaystyle\mu\|\phi\|_0^2 -\displaystyle\frac{1}{8}\|\nabla\phi\|_0^2 -\displaystyle\frac{1}{8}\|\nabla\phi\|_0^2 -\displaystyle\frac{1}{8}\|\nabla\phi\|_0^2-8M\Big(\sum\limits_{k=1}^{M}Z_K^2\Big)\|\phi\|_0^2
    \\\hskip 10em - 8\|\nabla\Psi_h\|_0^2\|\Psi_h\|_0^2\|\phi\|_0^2 -8C_{em}^2(\|\nabla\Psi_h\|_{0}^2+\|\nabla\Psi_h\|_0^{2\alpha})^2\|\phi\|_0^2 - a_2\|\phi\|_0^2.
  \end{array}
\end{equation}
Then we derive the desired result (\ref{coer}) by choosing
$$\mu= 8M\Big(\sum\limits_{k=1}^{M}Z_K^2\Big)+ 8\|\nabla\Psi_h\|_0^2\|\Psi_h\|_0^2
+8C_{em}^2(\|\nabla\Psi_h\|_{0}^2+\|\nabla\Psi_h\|_0^{2\alpha})^2 + a_2,$$ and the proof is
completed.
\end{proof}

\begin{theorem}\label{error_one_correction_Theorem}
  Under Assumption A-C, the eigenpair approximation
  $(\Lambda_{h_k}^{(\ell+1)},\Psi_{h_k}^{(\ell+1)})$ derived by by Algorithm
  \ref{MMAA} has following error estimates
\begin{equation}
  \|\Psi-\Psi_{h_k}^{(\ell+1)}\|_1 \lesssim |\Lambda-\Lambda_{h_k}^{(\ell)}|+\|\Psi-\Psi_{h_k}^{(\ell)}\|_{0}+\|\Psi-\mathbb P_{h_k}\Psi\|_1,\label{Error_Eigenfunction_H1_k1}
\end{equation}
\begin{equation}
  \|\Psi-\Psi_{h_k}^{(\ell+1)}\|_0+|\Lambda-\Lambda_{h_k}^{(\ell+1)}| \lesssim   r(V_{h_k})\|\Psi-\Psi_{h_k}^{(\ell+1)}\|_1,\label{Error_Eigenfunction_L2_k1}
\end{equation}
where the projection operator $\mathbb P_{h_k}:\mathcal H\rightarrow V_{h_k}$ is
defined by: $\mathbb P_{h_k}=(P_{h_k})^N$ and
\begin{eqnarray}\label{def4proj}
a_\mu(P_{h_k}u,v_{h_k})=a_\mu(u,v_{h_k}),\quad \forall v_{h_k}\in S_{h_k}, \ \forall u\in H_0^1(\Omega).
\end{eqnarray}
\end{theorem}
\begin{proof}
  From (\ref{Nonlinear_Eigenvalue_Problem2}) and (\ref{def4proj}), the following
  equation holds true for $i=1,\cdots, N$:
\begin{equation}\label{EQUA1}
a_{\mu}(P_{h_k} \psi_i,v_{h_k})=a_{\mu}(\psi_i,v_{h_k})
= \big( (\lambda_{i}+\mu) \psi_i,v_{h_k}\big)
-\big(V_{Har}(\rho_{\Psi})\psi_i,v_{h_k}\big)-\big(V_{xc}(\rho_{\Psi})\psi_i,v_{h_k}\big),\
\forall v_{h_k}\in S_{h_k}.
\end{equation}
Then from (\ref{Aux_Linear_Problem}) and (\ref{EQUA1}), we obtain
\begin{equation}\label{nonlinear analysis}
  \begin{array}{l}
a_{\mu}(P_{h_k} \psi_{i}-\widehat{\psi}_{i,h_k}^{(\ell+1)},v_{h_k}) + \Big(   (V_{Har}(\rho_{\Psi_{h_k}^{(\ell)}})+V_{xc}(\rho_{\Psi_{h_k}^{(\ell)}}))(P_{h_k} \psi_{i}-\widehat{\psi}_{i,h_k}^{(\ell+1)}),v_{h_k} \Big) \\
\hskip 2em =\big((\lambda_i+\mu)\psi_{i}-(\lambda_{i,h_k}^{(\ell)}+\mu)\psi_{i,h_k}^{(\ell)},v_{h_k}\big)
+(V_{Har}(\rho_{\Psi_{h_k}^{(\ell)}})P_{h_k}\psi_{i}, v_{h_k} )-\big(V_{Har}(\rho_{\Psi})\psi_{i}, v_{h_k}\big) \\
\hskip 4em +\big(V_{xc}(\rho_{\Psi_{h_k}^{(\ell)}})P_{h_k}\psi_{i},v_{h_k} \big) -\big(V_{xc}(\rho_{\Psi})\psi_{i}, v_{h_k} \big)
    , \quad \forall v_{h_k}\in S_{h_k}.
  \end{array}
\end{equation}
Now let us estimate the equation (\ref{nonlinear analysis}) termwise. We deduce
from triangle inequality that
\begin{equation}\label{first term}
\big((\lambda_i+\mu)\psi_{i}-(\lambda_{i,h_k}^{(\ell)}+\mu)\psi_{i,h_k}^{(\ell)},v_{h_k}\big)\nonumber \lesssim |\lambda_i- \lambda_{i,h_k}^{(\ell)}|\|\psi_i\|_0\|v_{h_k}\|_0
+(\lambda_{i,h_k}^{(\ell)}+\mu)\|\psi_{i}-\psi_{i,h_k}^{(\ell)}\|_0\|v_{h_k}\|_0.
\end{equation}
For the Hartree potential, using the uncertain principle lemma, the following
inequalities hold
\begin{equation}\label{D}
  \begin{array}{l}
\Big(V_{Har}(\rho_{\Psi_{h_k}^{(\ell)}})P_{h_k}\psi_{i}, v_{h_k} \Big)-\Big(V_{Har}(\rho_{\Psi})\psi_{i}, v_{h_k}\Big)\\\hskip 6em  =\Big(V_{Har}(\rho_{\Psi_{h_k}^{(\ell)}})(P_{h_k}\psi_i-\psi_i), v_{h_k}\Big)
+\Big((V_{Har}(\rho_{\Psi_{h_k}^{(\ell)}})-V_{Har}(\rho_{\Psi}))\psi_{i}, v_{h_k}\Big)\\
\hskip 6em =\big(r^{-1}*\rho_{\Psi_{h_k}^{(\ell)}},(P_{h_k}\psi_{i}-\psi_{i})v_{h_k}\big)
+\big(r^{-1}*(\rho_{\Psi_{h_k}^{(\ell)}}-\rho_{\Psi}),\psi_iv_{h_k}\big)\\
    \hskip 6em \lesssim \|r^{-1}*\rho_{\Psi_{h_k}^{(\ell)}}\|_{0,\infty}\|P_{h_k}\psi_{i}-\psi_{i}\|_0\|v_{h_k}\|_0 +\|r^{-1}*(\rho_{\Psi_{h_k}^{(\ell)}}-\rho_{\Psi})\|_{0,\infty}\|\psi_i\|_0\|v_{h_k}\|_0 \\
    \hskip 6em \lesssim \|\Psi_{h_k}^{(\ell)}\|_0\|\nabla\Psi_{h_k}^{(\ell)}\|_0\|P_{h_k}\psi_{i}-\psi_{i}\|_0\|v_{h_k}\|_0 +\|\Psi_{h_k}^{(\ell)}-\Psi\|_0\|\nabla(\Psi_{h_k}^{(\ell)}+\Psi)\|_0\|\psi_i\|_0\|v_{h_k}\|_0.
    \end{array}
\end{equation}
For any $\Gamma\in \mathcal H$, based on Assumption B, the exchange-correlation potential can be estimated by
\begin{equation}\label{echange correlation}
  \begin{array}{l}
    \Big(  V_{xc}(\rho_{\Psi_{h_k}}^{(\ell)})\mathbb P_{h_k}\Psi,\Gamma  \Big)-\Big(V_{xc}(\rho_{\Psi})\Psi, \Gamma\Big)\\
    \hskip 1em =\big(  V_{xc}(\rho_{\Psi_{h_k}}^{(\ell)})\Psi_{h_k}^{(\ell)},\Gamma  \big)-\big(V_{xc}(\rho_{\Psi})\Psi, \Gamma\big) +\big(  V_{xc}(\rho_{\Psi_{h_k}}^{(\ell)})(\mathbb P_{h_k}\Psi-\Psi_{h_k}^{(\ell)}),\Gamma  \big)\\
    \hskip 1em =2\Big(V'_{xc}(\rho_{\Psi_\xi})\displaystyle\sum_{i=1}^{N}\psi_{\xi,i}(\psi_{i,h_k}^{(\ell)}-\psi_{i}),
    \displaystyle\sum_{j=1}^{N}\psi_{\xi,j}\gamma_j\Big) +\Big(V_{xc}(\rho_{\Psi_\xi})(\Psi_{h_k}^{(\ell)}-\Psi),\Gamma\Big)\\\hskip 4em +\big(  V_{xc}(\rho_{\Psi_{h_k}}^{(\ell)})(\mathbb P_{h_k}\Psi-\Psi_{h_k}^{(\ell)}),\Gamma  \big)\\
    \hskip 1em \lesssim \Big(V'_{xc}(\rho_{\Psi_\xi})\rho_{\Psi_\xi}^{1/2}\big(\displaystyle\sum_{i=1}^{N}(\psi_{i,h_k}^{(\ell)}-\psi_{i})^2\big)^{1/2},
    \rho_{\Psi_{\xi}}^{1/2}\big(\displaystyle\sum_{j=1}^{N}\gamma_j^2\big)^{1/2}\Big) +\big(V_{xc}(\rho_{\Psi_\xi})(\Psi_{h_k}^{(\ell)}-\Psi),\Gamma\big)\\\hskip 4em +\big(  V_{xc}(\rho_{\Psi_{h_k}}^{(\ell)})(\mathbb P_{h_k}\Psi-\Psi_{h_k}^{(\ell)}),\Gamma  \big)\\
    \hskip 1em\lesssim \Big(V'_{xc}(\rho_{\Psi_\xi})\rho_{\Psi_\xi}\big(\displaystyle\sum_{i=1}^{N}\big(\psi_{i,h_k}^{(\ell)}-\psi_{i}\big)^2\big)^{1/2},
    \big(\displaystyle\sum_{j=1}^{N}\gamma_j^2\big)^{1/2}\Big) +\big(V_{xc}(\rho_{\Psi_\xi})(\Psi_{h_k}^{(\ell)}-\Psi),\Gamma\big)\\\hskip 4em +\big(  V_{xc}(\rho_{\Psi_{h_k}}^{(\ell)})(\mathbb P_{h_k}\Psi-\Psi_{h_k}^{(\ell)}),\Gamma  \big)\\
    \hskip 1em\lesssim \Big((\rho_{\Psi_\xi}+\rho_{\Psi_\xi}^\alpha)\big(\displaystyle\sum_{i=1}^{N}(\psi_{i,h_k}^{(\ell)}-\psi_{i})^2\big)^{1/2},
    \big(\displaystyle\sum_{j=1}^{N}\gamma_j^2\big)^{1/2}\Big) +\big(V_{xc}(\rho_{\Psi_\xi})(\Psi_{h_k}^{(\ell)}-\Psi),\Gamma\big)\\\hskip 4em +\big(  V_{xc}(\rho_{\Psi_{h_k}}^{(\ell)})(\mathbb P_{h_k}\Psi-\Psi_{h_k}^{(\ell)}),\Gamma  \big),
  \end{array}
\end{equation}
where $\Psi_{\xi}=\xi\Psi_{h_k}^{(\ell)}+(1-\xi)\Psi$ with $\xi\in [0, 1]$.

For the first part of (\ref{echange correlation}), we have
\begin{equation}\label{first of ex}
  \begin{array}{l}
\Big((\rho_{\Psi_\xi}+\rho_{\Psi_\xi}^\alpha)\big(\displaystyle\sum_{i=1}^{N}(\psi_{i,h_k}^{(\ell)}-\psi_{i})^2\big)^{1/2},
\big(\displaystyle\sum_{j=1}^{N}\gamma_j^2\big)^{1/2}\Big)\\
\hskip 8em \lesssim \Big((\rho_{\Psi_\xi}+\rho_{\Psi_\xi}^\alpha)\Big(\displaystyle\sum_{i=1}^{N}(\psi_{i,h_k}^{(\ell)}-\psi_{i})^2\Big)^{1/2}, \displaystyle\sum_{j=1}^{N}|\gamma_j|\Big) \\
    \hskip 8em \lesssim \displaystyle\sum_{j=1}^{N}\|\rho_{\Psi_\xi}^\alpha\|_{0,3/\alpha}\Big\|\Big(\displaystyle\sum_{i=1}^{N}(\psi_{i,h_k}^{(\ell)}-\psi_{i})^2\Big)^{1/2}\Big\|_{0,\Omega}\|\gamma_j\|_{0,6/(3-2\alpha)} \\
    \hskip 10em +\displaystyle\sum_{j=1}^{N}\|\rho_{\Phi_\xi}\|_{0,3}
\Big\|\Big(\displaystyle\sum_{i=1}^{N}(\psi_{i,h_k}^{(\ell)}-\psi_{i})^2\Big)^{1/2}\Big\|_{0,\Omega}\|\gamma_j\|_{0,6} \\
\hskip 8em \lesssim \|\Psi_{h_k}^{(\ell)}-\Psi\|_{0,\Omega}\|\Gamma\|_{1,\Omega}.
  \end{array}
\end{equation}
For the second and third parts of (\ref{echange correlation}), using the proof
technique for (\ref{mvd}), the following estimate holds
\begin{equation}\label{eexx}
\big(V_{xc}(\rho_{\Psi_\xi})(\Psi_{h_k}^{(\ell)}-\Psi),\Gamma\big)+\big(  V_{xc}(\rho_{\Psi_{h_k}}^{(\ell)})(\mathbb P_{h_k}\Psi-\Psi_{h_k}^{(\ell)}),\Gamma  \big) \lesssim (\|\Psi-\Psi_{h_k}^{(\ell)}\|_{0,\Omega}+\|\Psi-\mathbb P_{h_k}\Psi\|_{0,\Omega})\|\Gamma\|_{1,\Omega}.
\end{equation}
Using (\ref{echange correlation}), (\ref{first of ex}) and (\ref{eexx}), there holds
\begin{eqnarray}\label{e1}
\Big(  V_{xc}(\rho_{\Psi_{h_k}}^{(\ell)})\mathbb P_{h_k}\Psi,\Gamma  \Big)-\Big(V_{xc}(\rho_{\Psi})\Psi, \Gamma\Big)
\lesssim \Big(\|\Psi-\Psi_{h_k}^{(\ell)}\|_{0,\Omega}+\|\Psi-\mathbb P_{h_k}\Psi\|_{0,\Omega}\Big)\|\Gamma\|_{1,\Omega}.
\end{eqnarray}

Taking $v_{h_k} = P_{h_k} \psi_{i}-\widehat{\psi}_{i,h_k}^{(\ell+1)}$ in
(\ref{nonlinear analysis}), due to Lemma \ref{coercive lemma}, (\ref{D}), (\ref{echange correlation}) and (\ref{e1}), we derive
\begin{equation}\label{1^2}
  \begin{array}{lll}
\|\mathbb P_{h_k} \Psi-\widehat{\Psi}_{h_k}^{(\ell+1)}\|_1^2
&=&\displaystyle\sum_{i=1}^{N}\|P_{h_k} \psi_{i}-\widehat{\psi}_{i,h_k}^{(\ell+1)}\|_1^2\\
&\lesssim& \displaystyle\sum_{i=1}^{N}\Big\{a_{\mu}(P_{h_k} \psi_{i}-\widehat{\psi}_{i,h_k}^{(\ell+1)},P_{h_k} \psi_{i}-\widehat{\psi}_{i,h_k}^{(\ell+1)}) \\
&&\hskip 1em + (V_{Har}(\rho_{\Psi_{h_k}^{(\ell)}})(P_{h_k} \psi_{i}-\widehat{\psi}_{i,h_k}^{(\ell+1)}),P_{h_k} \psi_{i}-\widehat{\psi}_{i,h_k}^{(\ell+1)})\\
&&\hskip 2em + (V_{xc}(\rho_{\Psi_{h_k}^{(\ell)}})(P_{h_k} \psi_{i}-\widehat{\psi}_{i,h_k}^{(\ell+1)}),P_{h_k} \psi_{i}-\widehat{\psi}_{i,h_k}^{(\ell+1)})\Big\}\\
    &\lesssim& \Big(|\Lambda-\Lambda_{h_k}^{(\ell)}|+\|\Psi-\Psi_{h_k}^{(\ell)}\|_0+\|\Psi-P_{h_k} \Psi\|_0\Big)\|\mathbb P_{h_k} \Psi-\widehat{\Psi}_{h_k}^{(\ell+1)}\|_1,
    \end{array}
\end{equation}
which yields
\begin{eqnarray}\label{Inequality_1}
\|\Psi-\widehat{\Psi}_{h_k}^{(\ell+1)}\|_1
&\lesssim& \|\Psi-\mathbb P_{h_k} \Psi\|_1+\|P_{h_k}\Psi-\widehat{\Psi}_{h_k}^{(\ell+1)}\|_1 \nonumber\\
&\lesssim& \|\Psi-\mathbb P_{h_k} \Psi\|_1+|\Lambda-\Lambda_{h_k}^{(\ell)}|+ \|\Psi-\Psi_{h_k}^{(\ell)}\|_0.
\end{eqnarray}
Since $S_{H,h}$ is a subspace of $S_{h_k}$, the Kohn-Sham equation (\ref{Nonlinear_Eig_Hh}) can be regarded as a subspace
approximation for (\ref{Nonlinear_Eigenvalue_Problem2 fem}).  From Assumption C
and (\ref{Inequality_1}), we have
\begin{eqnarray*}
\|\Psi-\Psi_{h_k}^{(\ell+1)}\|_1\lesssim  \|\Psi-\widehat{\Psi}_{h_k}^{(\ell+1)}\|_1
\lesssim  \|\Psi-\mathbb P_{h_k} \Psi\|_1+|\Lambda-\Lambda_{h_k}^{(\ell)}|+ \|\Psi-\Psi_{h_k}^{(\ell)}\|_0.
\end{eqnarray*}
This is the desired result (\ref{Error_Eigenfunction_H1_k1}).  The estimate
(\ref{Error_Eigenfunction_L2_k1}) can be obtained by Assumption C, and the proof
is completed.
\end{proof}

Theorem \ref{error_one_correction_Theorem} shows that the augmented subspace
method defined by Algorithm \ref{MMAA} can really improve the accuracy after
each iteration, which is the base for designing the nonnested augmented subspace method
for Kohn-Sham equation.

Next, we estimate the computational work of the nonnested augmented subspace method.  In
this paper, we use Mshmet and Mmg3d \cite{Dapogny,Frey} to realize the moving
mesh process. Therefore, the unit mesh principle shall be met
\cite{Dapogny,Frey}.  That is, for any $K\in\mathcal T_h$, there holds
\begin{eqnarray}\label{union}
e^{\rm T}\mathcal M(k)e=1,\quad \forall e\in E_K,
\end{eqnarray}
where $\mathcal M(k)$ is the metric on the mesh element $K$, $E_K$ is the union
of the boundaries of $K$.  Based on the definition of $\mathcal M$ and the unit
mesh principle (\ref{union}), we can change the length of each edge through
adjusting the value of $\varepsilon$, which then leads to the following
relationship for the number of degrees of freedom:
\begin{eqnarray}\label{prop}
N_k\approx \eta^{n-k}N_n,\quad k=1,2\cdots,n,
\end{eqnarray}
where $\eta$ denotes the coarsening rate between two consecutive meshes.

Let $W_k$ denote the computational work in the $k-$th finite element space,
$\mathcal O(M_{k_0})$ denote the computational work for solving the Kohn-Sham
equation in the first $k_0$ adaptive spaces, $n$ denote the number of mesh
levels, $\varpi$ denote the nonlinear iteration times, $\vartheta$ denote the
number of processes. We use parallel computing technique to assemble the matrix
in step 3 of Algorithm \ref{MMAA}, and the resultant algebraic system is solved by
the parallel algebraic multigrid method. Thus the computational work for the $N$
linear boundary value problems is $\mathcal O(NN_k/\vartheta)$.  For the
small-scale nonlinear eigenvalue problem involved in step 4 of Algorithm
\ref{MMAA}, the matrices are assembled through parallel computing technique, and
the small-scale algebraic eigenvalue problem is solved in serial computing with
computational work $\mathcal O(M_H)$.
Then the total computational work in the
$k-$th $(k>k_0)$ finite element space is
\begin{eqnarray}\label{workonelevel}
W_k = \mathcal O \Big(  N\displaystyle\frac{N_k}{\vartheta} +\varpi \Big(\displaystyle\frac{N_k}{\vartheta}+M_H\Big)   \Big).
\end{eqnarray}

\begin{theorem}
  Assume that the number of degrees of freedom on two consecutive meshes
  satisfies the proportional relationship (\ref{prop}).  Then the total
  computational work of Algorithm \ref{MMAA} can be estimated as follows
\begin{eqnarray}
W_{total} =\mathcal O \Big( \displaystyle\frac{N+\varpi}{\vartheta} N_n +  \varpi {\rm log} N_n M_H +M_{k_0} \Big).
\end{eqnarray}
\end{theorem}

\begin{proof}
  Based on the computational work (\ref{workonelevel}) in one finite element space
  and the proportional relationship (\ref{prop}), we have
\begin{eqnarray}
W_{total} &=&\sum_{k=1}^N W_k=\mathcal O \left(  M_{k_0} +  \sum_{k=k_0+1}^n\Big( N\displaystyle\frac{N_k}{\vartheta} +\varpi (\displaystyle\frac{N_k}{\vartheta}+M_H) \Big) \right)\nonumber\\
&=&\mathcal O \left( \displaystyle\frac{N+\varpi}{\vartheta} \sum_{k=k_0+1}^n N_k +  \varpi {\rm log} N_n M_H +M_{k_0} \right) \nonumber\\
&=&\mathcal O \left( \displaystyle\frac{N+\varpi}{\vartheta} N_n +  \varpi {\rm log} N_n M_H +M_{k_0} \right).
\end{eqnarray}
Then the proof is completed.
\end{proof}

\begin{remark}
  For complicated molecular system, the number of the desired eigenpairs $(N)$
  and the number of nonlinear iteration times $(\varpi)$ will be large. But
  fortunately, the software package FreeFEM \cite{Jolivet} has a good  scalability. Thus, when the number of processes $\vartheta$ is large, we can
  still derive an asymptotically optimal computational work estimate.
\end{remark}

\section{Numerical experiments}

In this section, we provide four numerical examples to validate the efficiency
of the proposed nonnested augmented subspace method in this paper. These numerical
examples are implemented on the high performance computing platform LSSC-IV in the State Key
Laboratory of Scientific and Engineering Computing, Chinese Academy of
Sciences. Each computing node has two 18-core Intel Xeon Gold 6140 processors at
2.3 GHz and 192 GB memory. All the parallel solving processes in this paper use
72 cores.  We adopt the open source finite element software package FreeFEM
\cite{Hecht,Jolivet} to do the numerical simulation.  The anisotropic
measurement is generated by the Meshmet of Hessian matrix based on adaptive
function (density function), and the adaptive mesh is generated by the software
package Mmg3d \cite{Dapogny,Frey}.  The involved linear boundary value problems
are solved by PETSc-GAMG \cite{Balay}, and the linear eigenvalue problems are
solved by Krylov-Schur method \cite{Hernandez} in SLEPc \cite{Roman}.  In all
the numerical experiments, we set the threshold in Algorithm
\ref{One_Correction_Step} to be $K_{asm}=4$, and set the positive constant $\mu$ in
Algorithm \ref{MMAA} to be $\mu=8M\big(\sum_{k=1}^{M}Z_K^2\big)$.

For the nonlinear eigenvalue problems, we use the Anderson mixing scheme to
accelerate the convergence rate \cite{Anderson}. The detailed process can be described as follows:
Let $\rho_{h,in}^{(m)}$ and
$\rho_{h,out}^{(m)}$ represent the input and output electron densities of the
$m$-th self-consistent iteration.  To input the $(m+1)$-th self-consistent
iteration, $\rho_{h,in}^{(m+1)}$ is computed as follows:
\begin{eqnarray}
\rho_{h,in}^{(m+1)} = \beta \widetilde \rho_{h,out} +(1-\beta) \widetilde \rho_{h,in},
\end{eqnarray}
where
\begin{eqnarray}\label{mix2}
\widetilde \rho_{h,in(out)} = \alpha_m \rho_{h,in(out)}^{(m)} +\sum_{j=m_0}^{m-1} \alpha_j\rho_{h,in(out)}^{(j)},
\end{eqnarray}
and the sum of the coefficients equals to one, i.e.,
\begin{eqnarray}\label{mix3}
\alpha_{m_0} +\alpha_{m_0+1} + \cdots+\alpha_m=1.
\end{eqnarray}
According to (\ref{mix3}), the equation (\ref{mix2}) can be written as
\begin{eqnarray}\label{mix4}
\widetilde \rho_{h,in(out)} =  \rho_{h,in(out)}^{(m)} +\sum_{j=m_0}^{m-1} \alpha_j\big(\rho_{h,in(out)}^{(j)}-\rho_{h,in(out)}^{(m)}\big),
\end{eqnarray}
Here the positive integer $m_0\in [1,m]$, and the depth of the mixing scheme is
defined by depth$:=m-m_0+1$, $\beta$ is a given weight value.

Denoting $F^{(m)}=\rho_{h,out}^{(m)}-\rho_{h,in}^{(m)}$ and $\widetilde F=\widetilde \rho_{h,out}-\widetilde \rho_{h,in}$. Based on (\ref{mix4}), there holds
\begin{eqnarray}\label{mix5}
\widetilde F =  F^{(m)} +\sum_{j=m_0}^{m-1} \alpha_j\big(F^{(j)}-F^{(m)}\big).
\end{eqnarray}
The coefficients $\alpha_j$ are determined by minimizing
$\|\widetilde F\|_2^2=\|\widetilde \rho_{h,in}-\widetilde \rho_{h,out}\|_2^2$,
which amounts to solve the following equations:
\begin{eqnarray}\label{mix6}
\sum_{k=m_0}^{m-1}\big(F^{(m)}-F^{(j)},F^{(m)}-F^{(k)}\big)\alpha_k=\big(F^{(m)}-F^{(j)},F^{(m)}\big),\quad j=m_0,\cdots,m-1.
\end{eqnarray}
In our numerical experiments, we choose the depth$=5$, the damping coefficient
$\beta=0.7$.

\subsection{Kohn-Sham equation for Helium}\label{Example_He}

In the first example, we solve the following Kohn-Sham equation for Helium:
\begin{equation}\label{He}
\left\{
\begin{array}{l}
\Big(-\displaystyle \frac{1}{2}\Delta -\displaystyle\frac{2}{|x|} + \displaystyle\int_{\Omega} \frac{\rho(y)}{|x-y|}dy+V_{xc}\Big)\psi=
  \lambda \psi,\quad \text{in } \Omega,\\
  \\
  \psi= 0,\quad  \text{on } {\partial\Omega},\\
  \\
\displaystyle \int_{\Omega}|\psi|^2dx=1,
\end{array}
\right.
\end{equation}
where $\Omega=(-10,10)^3$. In this example, the electron density
$\rho = 2|\psi|^2$.  The exchange-correlation potential is chosen according to
(\ref{expot}) and (\ref{copot}).  The experimental value of non-relativistic
ground state energy of Helium atom is -2.90372 hartree \cite{Veillard}.

Here, we use three types of algorithms to solve the Kohn-Sham equation
(\ref{He}).  The first one uses the self-consistent field iteration to solve
(\ref{He}) directly on the fixed structure mesh.  The second one uses the
self-consistent field iteration to solve (\ref{He}) directly on the
adaptive moving mesh (i.e. step 3 of Algorithm \ref{One_Correction_Step} is solved
directly by the self-consistent field iteration).  The third one uses the
nonnested augmented subspace method to solve (\ref{He}).  In the moving mesh adaptive
process, the metric matrix is given by the modified Hessian matrix (\ref{m1}) of
$\rho^{1/2}(x)$.  The tolerance in Algorithm \ref{One_Correction_Step} is set to
be $1E-3$.  The tolerance in Algorithm \ref{MMAA} is set to be $2E-4$.

\begin{figure}[!htbp]
\centering
\includegraphics[width=5.5cm]{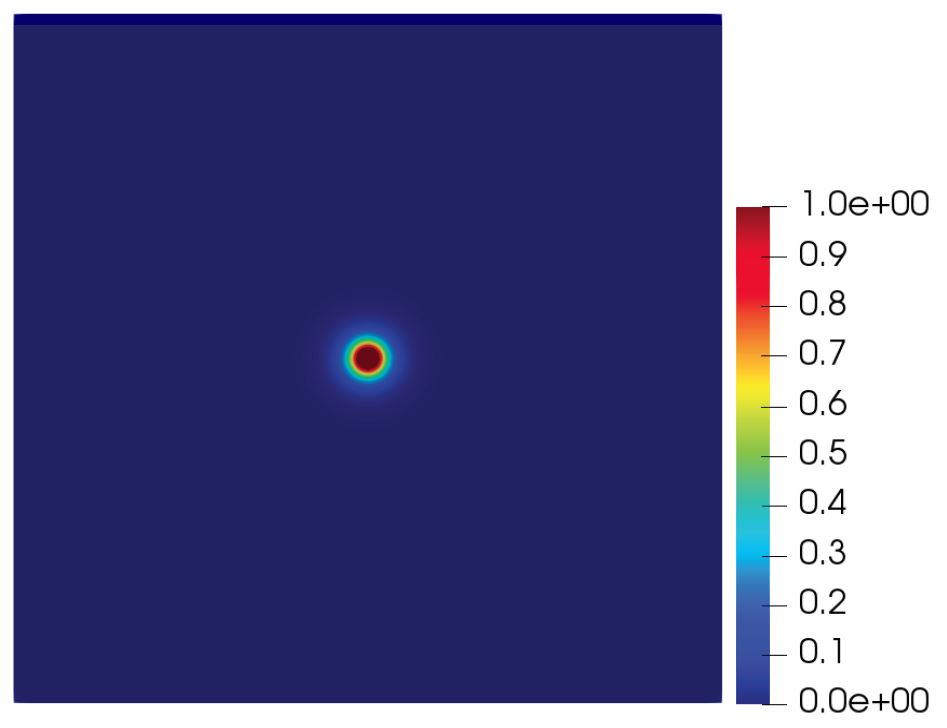} \ \ \ \ \ \ \ \ \ \
\includegraphics[width=4.3cm]{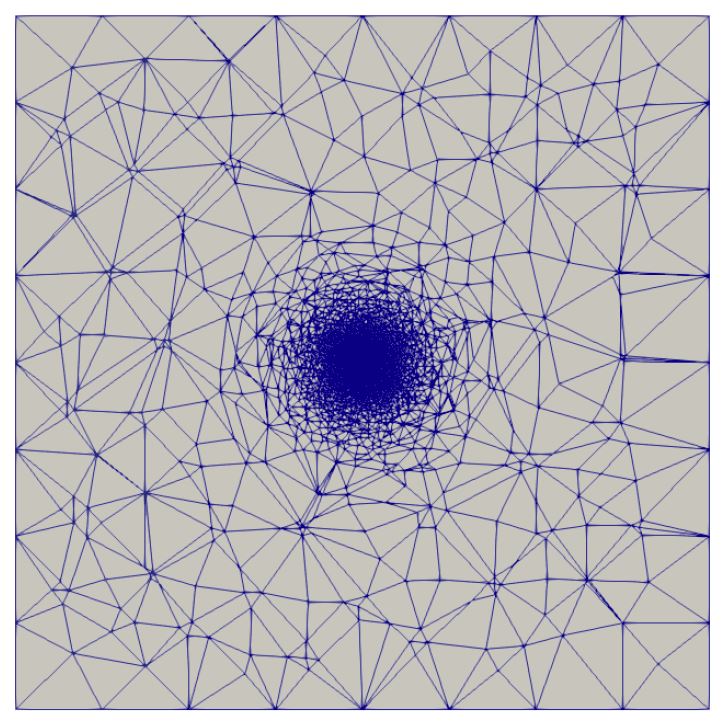}
\caption{Contour plot of the electron density (left) and the adaptive moving mesh (right) for Example 1}\label{hemesh}
\end{figure}

\begin{figure}[!htbp]
\centering
\includegraphics[width=6.0cm,height=4.5cm]{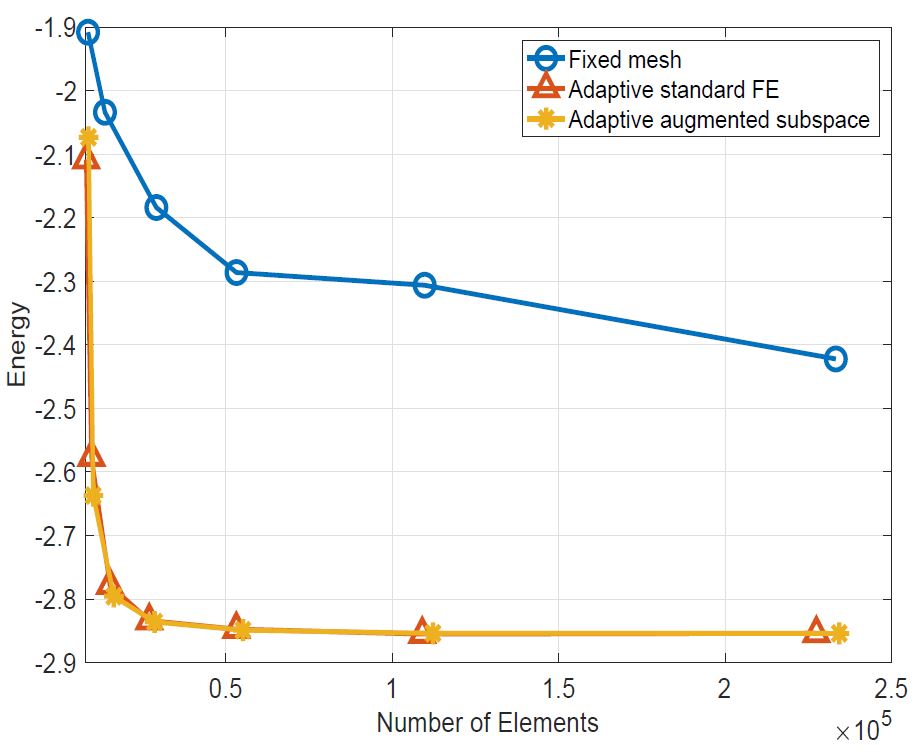} \ \ \ \ \
\includegraphics[width=6.0cm,height=4.5cm]{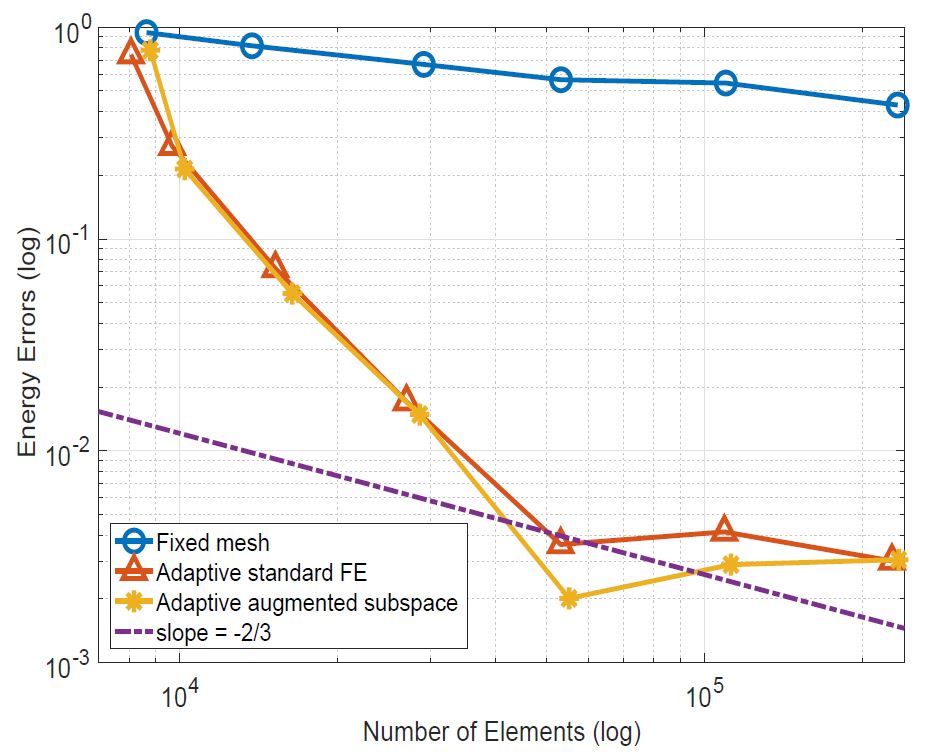}
\caption{The ground state energy (left) and the corresponding error estimates (right) for Example 1}\label{heerror}
\end{figure}

The density function and the adaptive mesh of the nonnested augmented subspace method are
presented in Figure \ref{hemesh}.  The ground state energy and the corresponding
error estimates derived by these three methods are presented in Figure
\ref{heerror}.  It can be seen from Figure \ref{heerror} that the ground state
energy of Helium atom first decreases quickly and then tends to be stable with
the adaptive refinement of the mesh.  Besides, we also can see that the adaptive
iterative methods have a better accuracy than the finite element method with the
uniform mesh.

\subsection{Kohn-Sham equation for Hydrogen-Lithium}\label{KS_HL}
In the second example, we solve the following Kohn-Sham equation for
Hydrogen-Lithium:
\begin{equation}\label{HLi}
\left\{
\begin{array}{l}
\Big(-\displaystyle\frac{1}{2}\Delta -\displaystyle\frac{3}{|x-r_1|}-\displaystyle\frac{1}{|x-r_2|}+
  \displaystyle\int_{\Omega}\frac{\rho(y)}{|x-y|}dy+V_{xc}\Big)\psi_i= \lambda_i\psi_i, \ \ \text{in } \Omega,\  \ \ \   i=1,2,\\
  \\
  \psi_i= 0,\ \ \ \ \ \ \text{on } {\partial\Omega}, \ \ i=1,2,\\
  \\
\displaystyle\int_\Omega \psi_i\psi_jd\Omega=\delta_{i,j},\ \ \ i,j=1,2,
\end{array}
\right.
\end{equation}
where $\Omega=(-10,10)^3$.  In this equation, the electron density
$\rho = 2(|\psi_1|^2+|\psi_2|^2)$, $r_j \ (j=1,2)$ is the position of lithium
atom and hydrogen atom and we take $r_1=(-1.0075,0,0)$, $r_2=(2.0075,0,0)$.  The
exchange-correlation potential is chosen according to (\ref{expot}) and
(\ref{copot}).  The experimental value of non relativistic ground state energy
of Hydrogen-Lithium molecular system is -8.070 hartree \cite{Cencek}.  In this example, we set the tolerance in Algorithm \ref{One_Correction_Step} to
be $1E-3$ and tolerance in Algorithm \ref{MMAA} to be $2E-4$.  We also use these
three numerical methods as described in the first example to solve (\ref{HLi}).

\begin{figure}[!htbp]
\centering
\includegraphics[width=5.5cm]{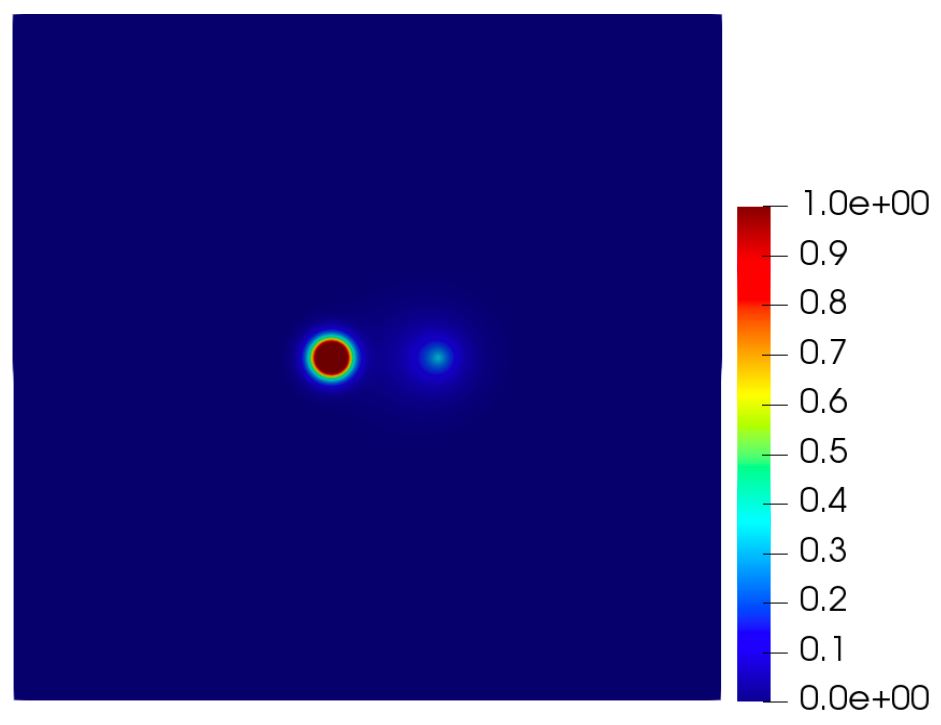} \ \ \ \ \ \ \ \ \ \
\includegraphics[width=4.3cm]{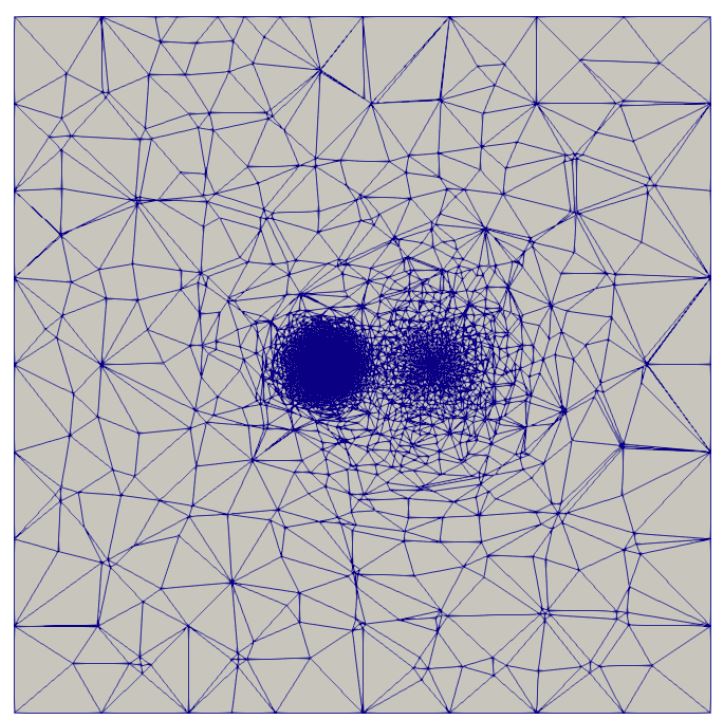}
\caption{Contour plot of the electron density (left) and the adaptive moving
  mesh (right) for Example 2}\label{hlimesh}
\end{figure}

\begin{figure}[!htbp]
\centering
\includegraphics[width=6cm,height=4.5cm]{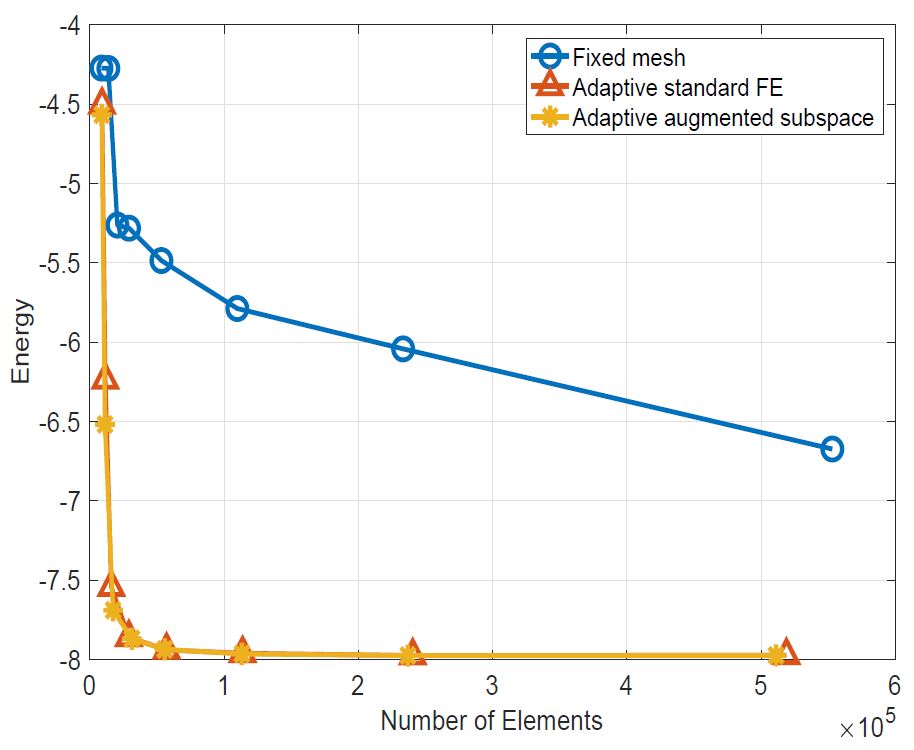} \ \ \ \ \
\includegraphics[width=6cm,height=4.5cm]{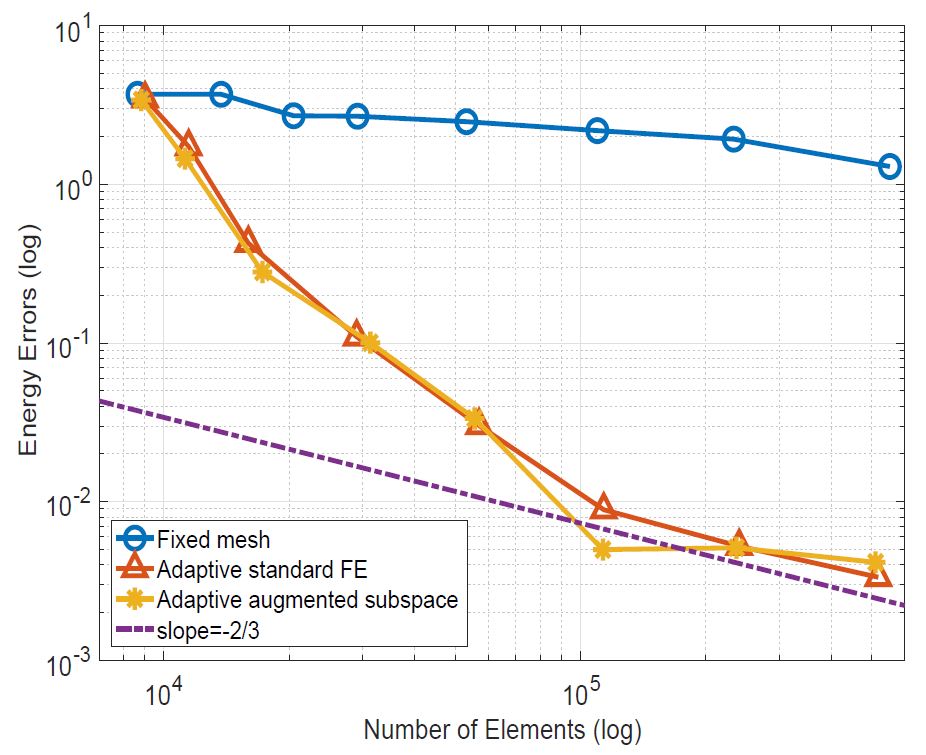}
\caption{The ground state energy (left) and the corresponding error estimates
  (right) for Example 2}\label{hlierror}
\end{figure}

The density function and the adaptive mesh for Hydrogen-Lithium are presented in
Figure \ref{hlimesh}.  The ground state energy and the corresponding error
estimates derived by these three methods are presented in Figure \ref{hlierror}.
It can be seen from Figure \ref{hlierror} that the ground state energy of
Hydrogen-Lithium decreases quickly to a stable constant with the adaptive
refinement of the mesh. Besides, we also can see that the adaptive iterative
methods have a better accuracy than the finite element method with uniform mesh.
At the same time, it can be seen that the convergence rate of Algorithm
\ref{MMAA} tends to $2/3$ with mesh refinement.

\begin{table}[htbp]
\begin{center}
\begin{tabular}{|c|c|c|c|c|c|c|c|c|}\hline
Elements (Direct method)& 9008 &  11463 &  15943  & 29015  & 57078  & 113729  & 240435  & 518786 \\ \hline
Time (Direct method) & 7.024  & 10.231  & 14.188  & 18.587  & {\bf 23.694}  & {\bf 32.339}  & {\bf 46.215}  & {\bf 74.595}  \\ \hline
Elements (Algorithm \ref{MMAA})&  8819 & 11240 & 17235 & 31297 & 55690 & 113168 & 236780 & 510949  \\ \hline
Time (Algorithm \ref{MMAA}) &  8.642  & 11.176  & 15.370  & 19.260  & {\bf 23.340}  & {\bf 30.560}  & {\bf 38.086}  & {\bf 49.016} \\ \hline
\end{tabular}
\end{center}
\caption{CPU time (in seconds) for the nonnested augmented subspace method and the direct
  adaptive finite element method.}
\label{time4hli}
\end{table}

In Table \ref{time4hli}, we present the computational time of the nonnested augmented
subspace method and the direct adaptive finite element method (i.e. solve the
Kohn-Sham equation directly using the moving mesh adaptive technique).  It can
be seen that the nonnested augmented subspace method is more efficient when the number of
mesh elements reaches 55690 or more, and the advantage is more obvious with the
increase of the number of mesh elements.

\subsection{Kohn-Sham equation for Methane}\label{KS_ch4}

In the third example, we consider the Kohm-Sham equation for Methane.  To show
the efficiency of the nonnested augmented subspace method intuitively, we also use these
three methods described in Example \ref{Example_He} to solve the Methane
molecules.  The computational domain is set to be $\Omega=(-10,10)^3$, and the
atomic positions of Methane are shown in Table \ref{position4c4h4}.
\begin{table}[htbp]
\begin{center}
\begin{tabular}{|c|c|c|c|c|}\hline
Atom & x & y & z & Nuclear charge \\ \hline
$\rm C_1$ & 0.0000   &  0.0000   & 0.0000   & 6   \\ \hline
$\rm H_2$ &  1.3092  &  1.3092   & 1.3092   & 1   \\ \hline
$\rm H_3$ &  -1.3092 &  -1.3092  & 1.3092   & 1   \\ \hline
$\rm H_4$ &  1.3092  &  -1.3092  & -1.3092  & 1   \\ \hline
$\rm H_5$ &  -1.3092 &  1.3092   & -1.3092  & 1   \\ \hline
\end{tabular}
\end{center}
\caption{The position and the nuclear charge of each atom in Methane.}
\label{position4c4h4}
\end{table}

For a full potential calculation, there are total ten electrons.  Since we don't
consider spin polarisation, five eigenpairs need to be calculated.  The
tolerance setting in Algorithms \ref{One_Correction_Step} and \ref{MMAA} is the
same as that of Example \ref{Example_He}.  The reference value of non
relativistic ground state energy of methane molecule is set to be -40.41 hartree
\cite{Johnson}.

Figure \ref{ch4mesh} shows the electron density distribution and adaptive mesh
of Methane molecule.  Figure \ref{ch4error} shows the change curve of ground
state energy and the corresponding error estimates for Methane molecule with the
increasement of the number of mesh elements.

\begin{figure}[!htbp]
\centering
\includegraphics[width=5.5cm]{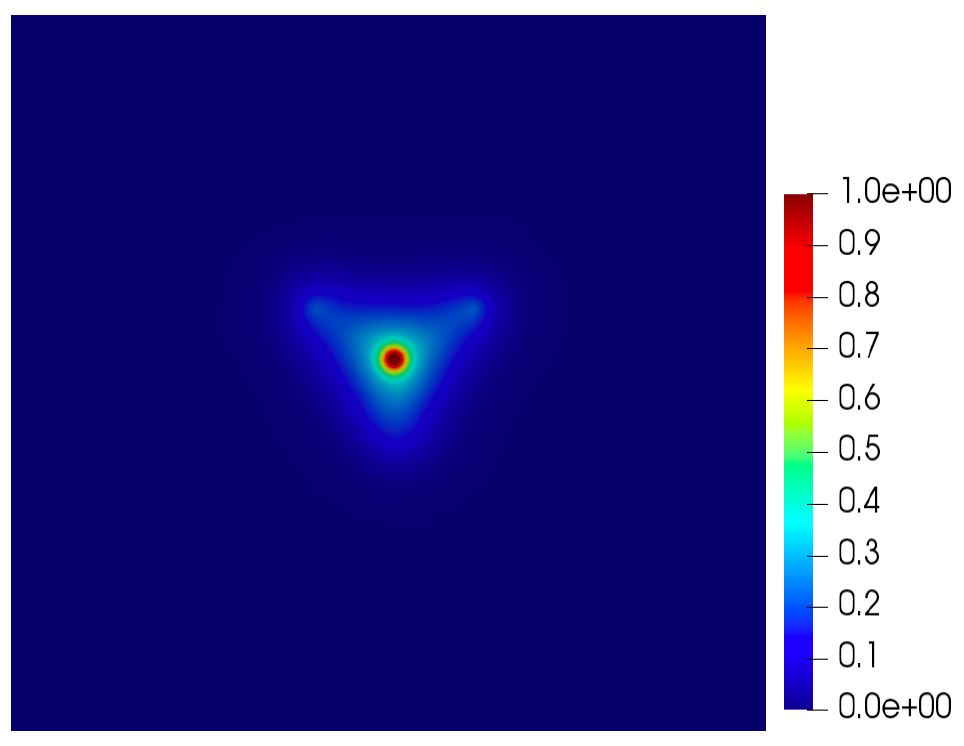} \ \ \ \ \ \ \ \ \ \
\includegraphics[width=4.3cm]{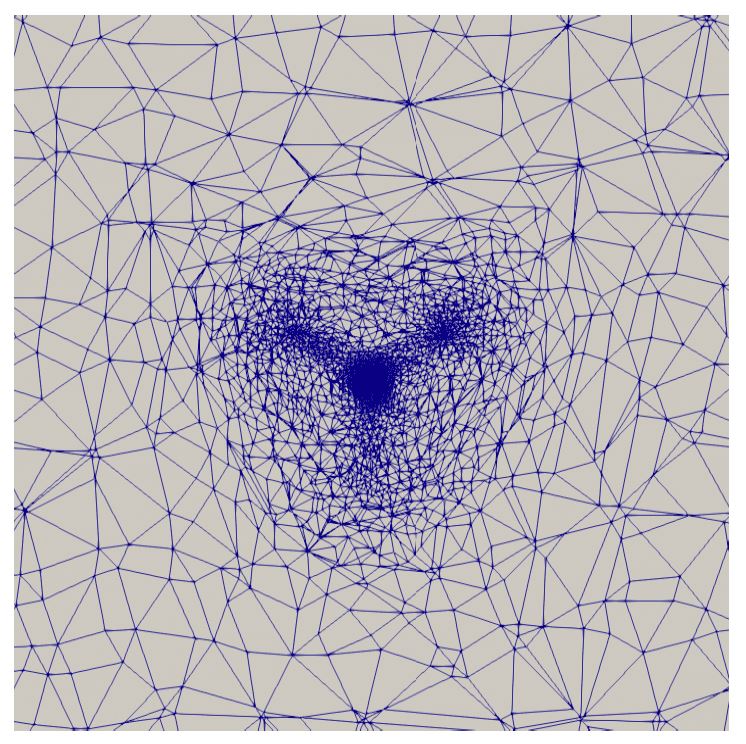}
\caption{Contour plot of the electron density (left) and the adaptive moving
  mesh (right) for Example 3}\label{ch4mesh}
\end{figure}

\begin{figure}[!htbp]
\centering
\includegraphics[width=5.5cm,height=5cm]{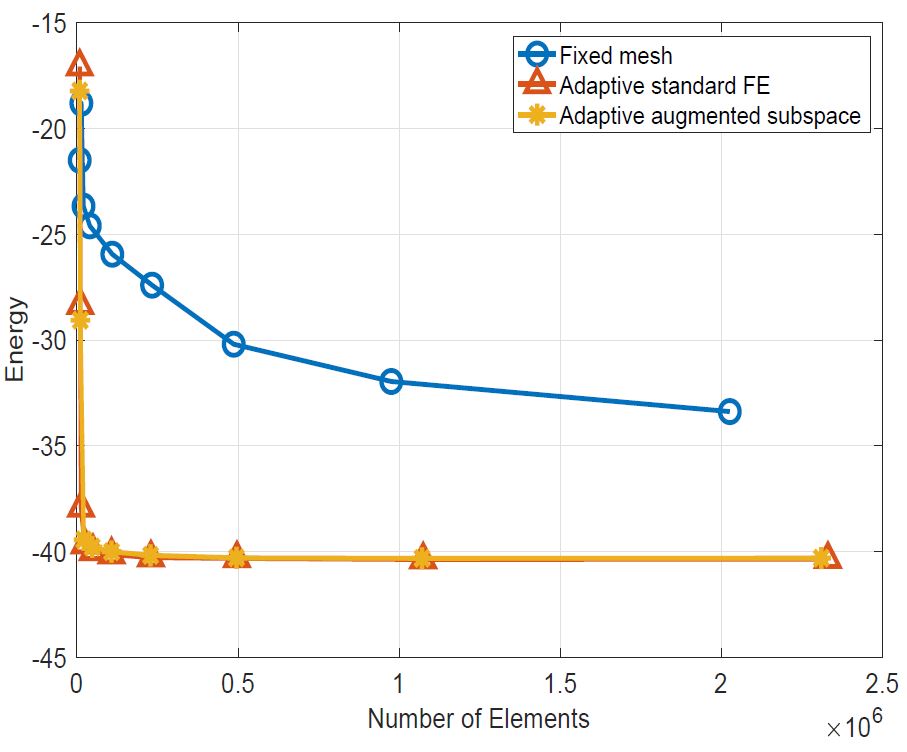} \ \ \ \ \ \ \ \ \ \
\includegraphics[width=5.5cm,height=5cm]{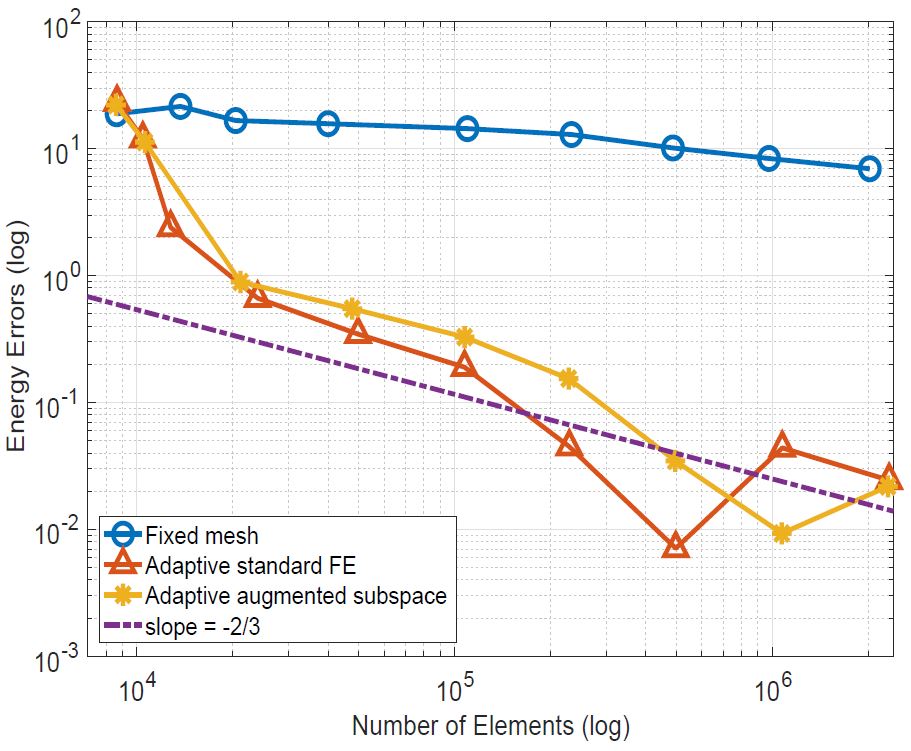}
\caption{The ground state energy (left) and the corresponding error estimates
  (right) for Example 3}\label{ch4error}
\end{figure}

Table \ref{time4c4h4} shows the computational time of the nonnested augmented subspace
method and the direct adaptive finite element method.  It can be seen that when
the number of mesh elements reaches 499064, the nonnested augmented subspace method has an
advantages in solving efficiency over the direct adaptive finite element method.

\begin{table}[htbp]
\begin{center}
\begin{tabular}{|c|c|c|c|c|c|c|c|c|c|}\hline
Elements (Direct method) & 10208 &  11496  & 22267  & 45982  & 102415  & 225074  & 492919  & 1072524  & 2329210 \\ \hline
Time (Direct method)  & 10.904 &  15.141 &  20.737 &  29.902 &  42.412  & 68.716  & {\bf 140.839}  & {\bf 312.876} &  {\bf 756.633}  \\ \hline
Elements (Algorithm \ref{MMAA}) & 10545  & 10599 &  22682 &  51646 &  108014  & 230730  & 499064 &  1083722 &  2342607 \\ \hline
Time (Algorithm \ref{MMAA})   & 10.762  & 14.027  & 19.386  & 30.579  & 53.005  & 70.812  & {\bf 86.976}  & {\bf 149.184}  & {\bf 209.735} \\ \hline
\end{tabular}
\end{center}
\caption{CPU time (in seconds) for the nonnested augmented subspace method and the direct
  adaptive finite element method.}
\label{time4c4h4}
\end{table}

\subsection{Kohn-Sham equation for Benzene}\label{KS_c6h6}

In the last example, we consider the Kohn-Sham equation for the Benzene
molecule.  The computational domain is taken as $\Omega=(-10,10)^3$, and the
atomic positions of Methane are shown in Table \ref{position4c6h6}.

\begin{table}[htbp]
\begin{center}
\begin{tabular}{|c|c|c|c|c|}\hline
Atom & x & y & z & Nuclear charge \\ \hline
$\rm C_1$    &  0.0000  & 1.3970   & 0.0000  & 6   \\ \hline
$\rm C_2$    &  1.2098  & 0.6985   & 0.0000  & 6   \\ \hline
$\rm C_3$    &  1.2098  & -0.6985  & 0.0000  & 6   \\ \hline
$\rm C_4$    &  0.0000  & -1.3970  & 0.0000  & 6   \\ \hline
$\rm C_5$    &  -1.2098 & -0.6985  & 0.0000  & 6   \\ \hline
$\rm C_6$    &  -1.2098 & 0.6985   & 0.0000  & 6   \\ \hline
$\rm H_7$    &  0.0000  & 2.4810   & 0.0000  & 1   \\ \hline
$\rm H_8$    &  2.1486  & 1.2405   & 0.0000  & 1   \\ \hline
$\rm H_9$    &  2.1486  & -1.2405  & 0.0000  & 1   \\ \hline
$\rm H_{10}$ &  0.0000  & -2.4810  & 0.0000  & 1   \\ \hline
$\rm H_{11}$ &  -2.1486 & -1.2405  & 0.0000  & 1   \\ \hline
$\rm H_{12}$ &  -2.1486 & 1.2405   & 0.0000  & 1   \\ \hline
\end{tabular}
\end{center}
\caption{The position and the nuclear charge number of each atom in Benzene.}
\label{position4c6h6}
\end{table}

For a full potential calculation, there are total 42 electrons.  Since we don't
consider spin polarisation, 21 eigenpairs need to be calculated.  The tolerance
setting in Algorithms \ref{One_Correction_Step} and \ref{MMAA} is the same as
that of Example \ref{Example_He}.  The reference value of non relativistic
ground state energy of methane molecule is set to be -231.78 hartree
\cite{Johnson}.

\begin{table}[htbp]
\begin{center}
\begin{tabular}{|c|c|c|c|c|c|c|c|c|}\hline
Elements (Direct method) & 8946  & 13095 &  21675  & 58635 &  198145 &  535501  & 1245081 &  2789529 \\ \hline
Time (Direct method)  & 19.079 &  28.454 &  40.958  & 63.495  & {\bf 109.606} &  {\bf 223.390} &  {\bf 470.353} &  {\bf 1027.426}  \\ \hline
Elements (Algorithm \ref{MMAA}) & 8897  & 12505 &  20535 &  62348 &  207362 &  550236  & 1262297  & 2796022 \\ \hline
Time (Algorithm \ref{MMAA})   & 14.235  & 21.239  & 35.367  & 62.337  & {\bf 79.789}  & {\bf 133.229} &  {\bf 187.119} &  {\bf 294.440} \\ \hline
\end{tabular}
\end{center}
\caption{CPU time (in seconds) for the nonnested augmented subspace method and the direct
  adaptive finite element method.}
\label{time4c6h6}
\end{table}

\begin{figure}[!htbp]
\centering
\includegraphics[width=5.5cm]{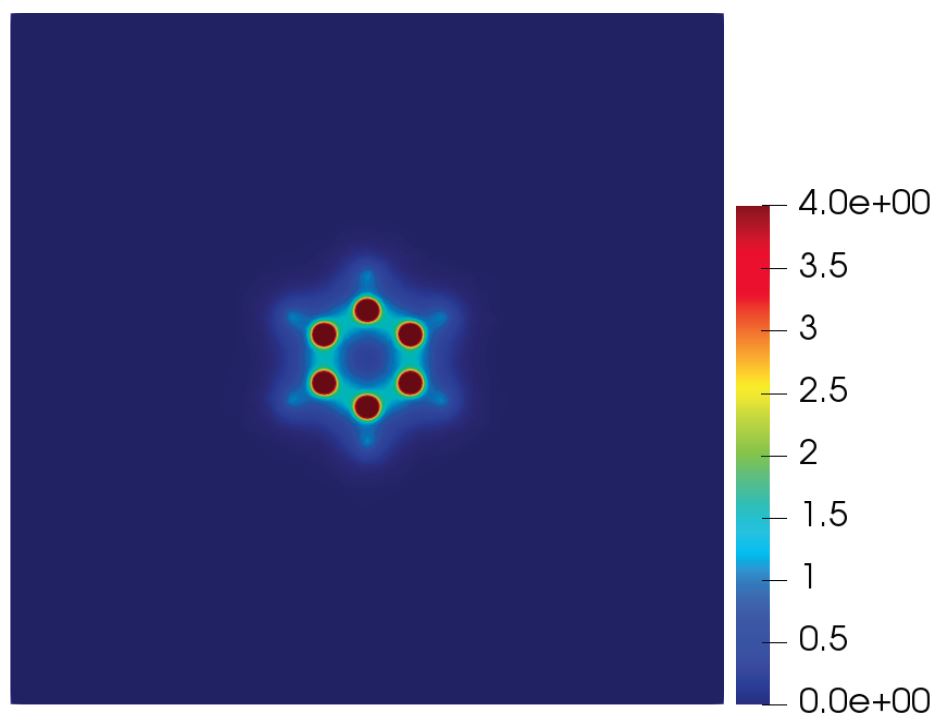} \ \ \ \ \ \ \ \ \ \
\includegraphics[width=4.3cm]{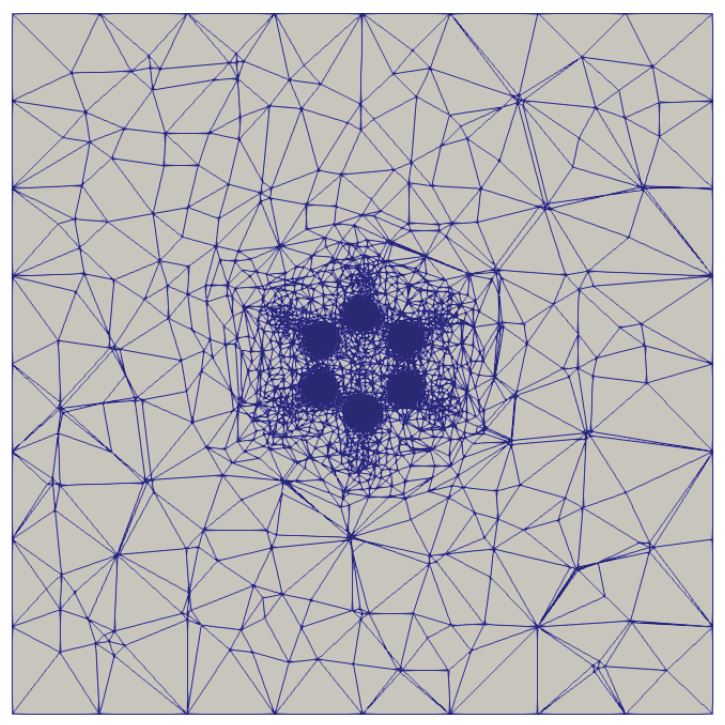}
\caption{Contour plot of the electron density (left) and the adaptive moving
  mesh (right) for Example 4}\label{c6h6mesh}
\end{figure}

Three numerical methods described in Example \ref{Example_He} are tested in
this example.  Figure \ref{c6h6mesh} shows the electron density and adaptive
mesh of the nonnested augmented subspace method.  It can be seen from Figure
\ref{c6h6mesh} that the mesh nodes gather in the region with singular electron
density (near the atomic centers), while in the region far away from the atomic
centers, the mesh node distribution is sparse, which is consistent with the
distribution of electron density function.

\begin{figure}[!htbp]
\centering
\includegraphics[width=5.5cm,height=5cm]{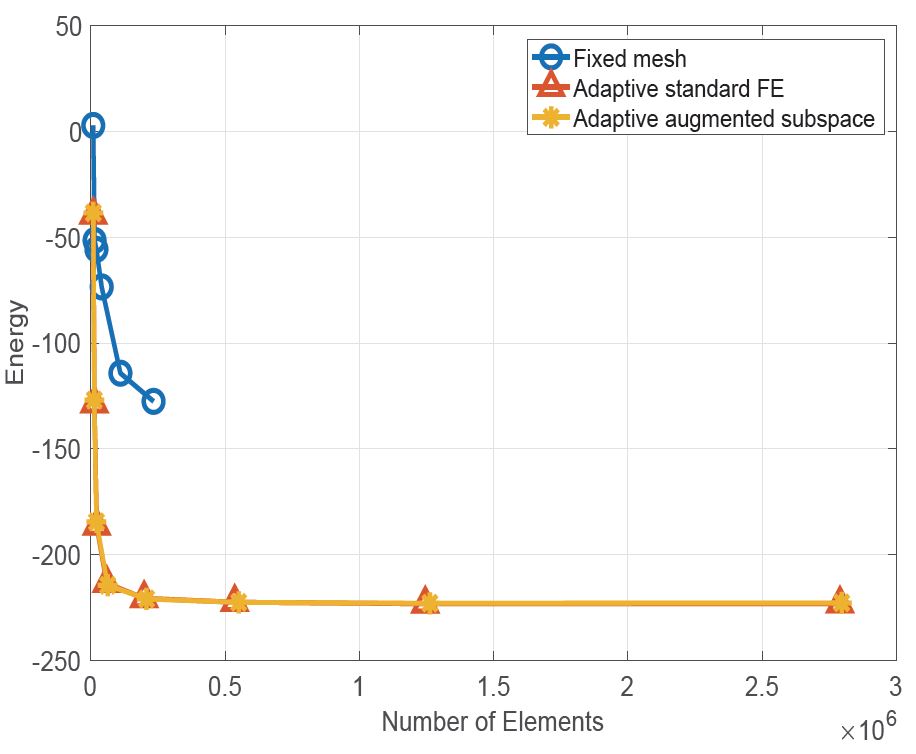} \ \ \ \ \ \ \ \ \ \
\includegraphics[width=5.5cm,height=5cm]{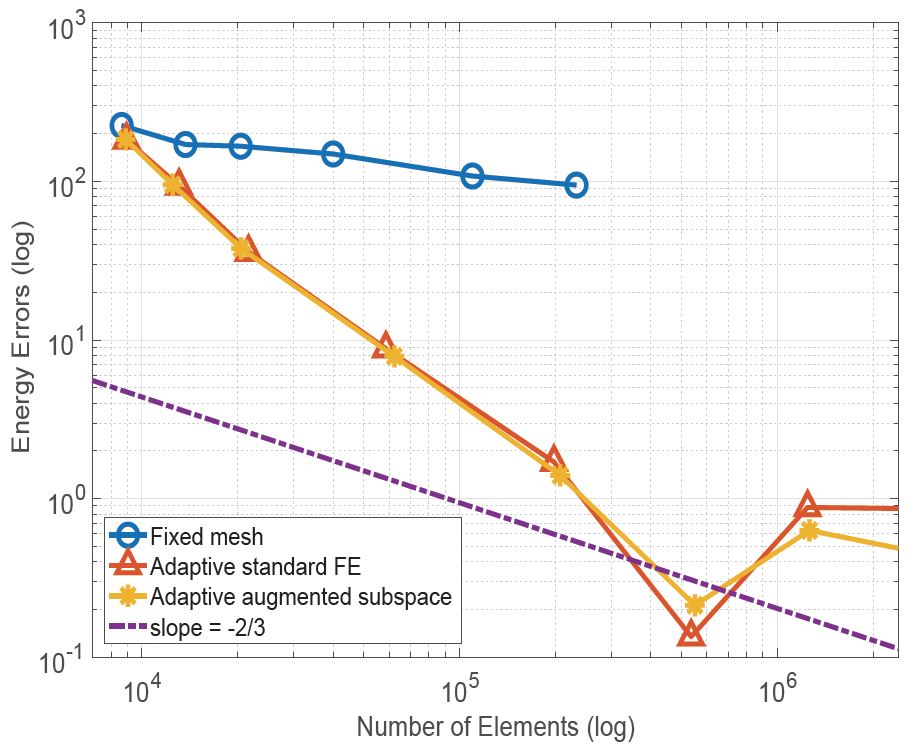}
\caption{The ground state energy (left) and the corresponding error estimates
  (right) for Example 4}\label{c6h6error}
\end{figure}

Figure \ref{c6h6error} shows the decline curve of ground state energy and
corresponding error estimates with mesh adaptive movement.  From Figure
\ref{c6h6error}, we can see that compared with the moving mesh adaptive method
(red line and yellow line), the finite element method performed on the fixed
structural mesh (blue line) derives a slower convergence rate.  Moreover, when
the number of mesh elements reaches 233280, the nonlinear iteration on fixed structure mesh still does
not converge even if the Anderson acceleration technology is adopted.

\begin{figure}[!htbp]
\centering
\includegraphics[width=5.5cm,height=5cm]{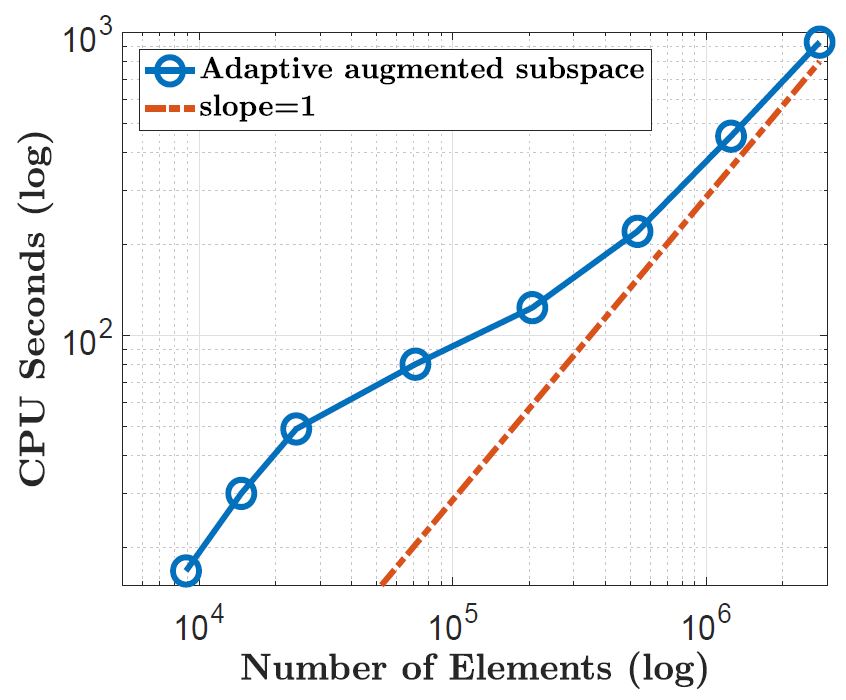} \ \ \ \ \ \ \ \ \ \
\includegraphics[width=5.5cm,height=5cm]{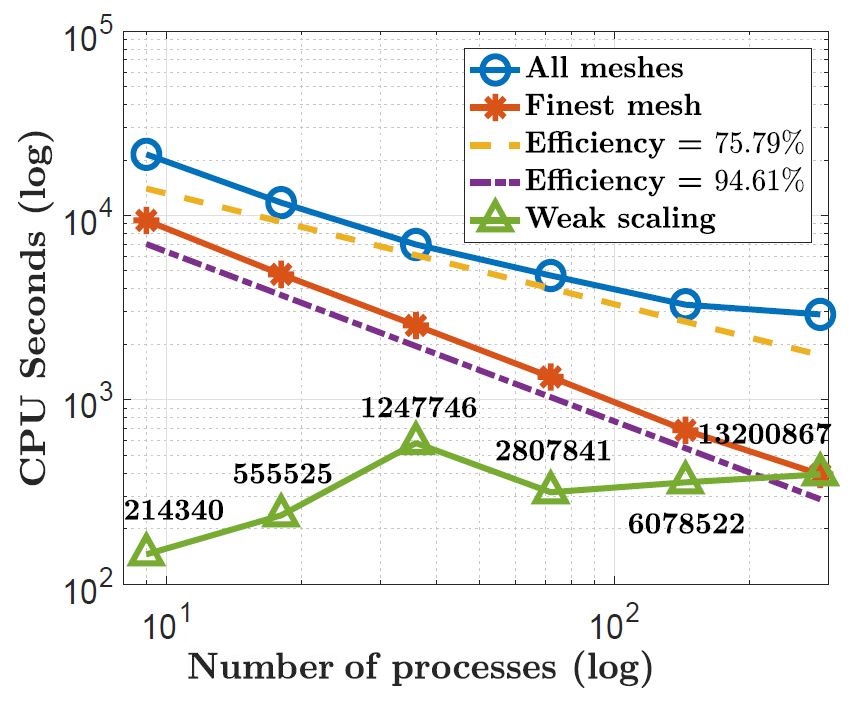}
\caption{The computational time and parallel scalability of the nonnested augmented
  subspace method for Example 4}\label{c6h6para}
\end{figure}

Table \ref{time4c6h6} compares the computational time of the nonnested augmented subspace
method and the direct adaptive finite element method.  It can be seen that when
the number of mesh elements reaches 207362, the nonnested augmented subspace method has advantage in solving efficiency over the direct adaptive finite element
method.

In addition, we also present the computational time of the nonnested augmented subspace
method for Benzene in Figure \ref{c6h6para} to test its linear complexity and
the parallel scalability.  Figure \ref{c6h6para} shows that when the number of
mesh elements reaches 500000, the nonnested augmented subspace method has a linear
computational complexity with the increase of the number of mesh elements.  The
strong scalability and weak scalability is also presented in Figure
\ref{c6h6para}.  In Figure \ref{c6h6para}, the label ``All meshes" (blue line)
represents the total computational time from the coarsest mesh to the finest
mesh.  The labeled ``finest mesh" (red line) represents only the computational
time in the finest mesh, The green line is the result of weak scalability, and
9, 18, 36, 72, 144 and 288 processes are used, respectively.  From Figure
\ref{c6h6para}, we can find that the nonnested augmented subspace method has a good scalability.

\section{Concluding remarks}
In this paper, a nonnested augmented subspace method is proposed to solve the
Kohn-Sham equation based on the moving mesh adaptive technique and augmented subspace method.
Because the moving mesh adaptive technology is used to
generate nonnested mesh sequence, the redistributed mesh is almost optimal. In
the solving process, we transform the classical nonlinear eigenvalue problem on
a fine mesh into a series of linear boundary value problems of the same scale on
the adaptive mesh sequence, and subsequently correct the approximate solutions by solving the
small-scale nonlinear eigenvalue problems in a special low-dimensional augmented
subspace.  Since solving large-scale nonlinear eigenvalue problem is avoided
which is time-consuming, the total solving efficiency can be greatly improved.
Some numerical experiments are also presented to verify the computational
efficiency and energy accuracy of the algorithm proposed in this paper.

\section{Acknowledgement}

This research is supported partly by National Key R\&D Program of China
2019YFA0709600, 2019YFA0709601, National Natural Science Foundations of China
(Grant Nos. 11922120, 11871489 and 11771434), FDCT of Macao SAR (0082/2020/A2),
MYRG of University of Macau (MYRG2020-00265-FST), the National Center for
Mathematics and Interdisciplinary Science, CAS, Beijing Municipal Natural Science Foundation (Grant No. 1232001),
and the General projects of science and technology plan of Beijing Municipal Education Commission (Grant No. KM202110005011).



\end{document}